\documentclass{amsart}
\usepackage[T1]{fontenc}

\usepackage{mathtools,amsfonts,amsthm,esint,braket,amsmath}
\usepackage{amsrefs}
\usepackage{enumitem}

\def\R{\mathbb{R}}
\def\Rd{\mathbb{R}^d}
\def\E{\mathcal{E}}
\def\F{\mathcal{F}}
\def\G{\mathcal{G}}
\def\H{\mathcal{H}}
\def\dist{\mathrm{dist}}

\renewcommand{\epsilon}{\varepsilon}
\newcommand{\eps}{\epsilon}
\renewcommand{\theta}{\vartheta}
\renewcommand{\phi}{\varphi}

\newtheorem{thm}{Theorem}
\newtheorem{prop}{Proposition}
\newtheorem{cor}{Corollary}
\newtheorem{lemma}{Lemma}
\newtheorem{rmk}{Remark}

\usepackage{hyperref}

\newcommand{\comment}[1]{}

\begin{document}
\title{On the  convergence rate of~some~nonlocal~energies}

\author[A. Chambolle]{Antonin Chambolle}
\author[M. Novaga]{Matteo Novaga}
\author[V. Pagliari]{Valerio Pagliari}
\address[A. Chambolle]{
CMAP, \'Ecole Polytechnique,
91128 Palaiseau Cedex,
France.
Email: \href{mailto:antonin.chambolle@cmap.polytechnique.fr }{\tt antonin.chambolle@cmap.polytechnique.fr}.
}
\address[M. Novaga]{
	Dipartimento di Matematica,
	Universit\`a di Pisa, 
	Largo B. Pontecorvo 5, 56127 Pisa, Italy.
	Email: \href{mailto:matteo.novaga@unipi.it}{\tt matteo.novaga@unipi.it}.
}
\address[V. Pagliari]{
	Institute of Analysis and Scientific Computing,
	TU Wien,
	Wiedner Hauptstra{\ss}e 8 - 10, 1040 Vienna, Austria.
	Email: \href{mailto:valerio.pagliari@tuwien.ac.at}{\tt valerio.pagliari@tuwien.ac.at}
}

\begin{abstract}
We study the rate of convergence of some nonlocal functionals
recently considered by Bourgain, Brezis, and Mironescu, 
and, after a suitable rescaling,
we establish the $\Gamma$-convergence 
of the corresponding rate functionals
to a limit functional of second order.
\end{abstract}

\maketitle

\tableofcontents

\section{Introduction}
We are interested in the rate of converge, as $h\searrow 0$, of the nonlocal functionals
\[
\F_h(u)\coloneqq \int_{\Rd}\int_{\Rd}K_h(y-x) f \left(\frac{\left| u(y)-u(x) \right|}{\left| y - x \right|}\right) dydx,
\]
to the limit functional 
	\[
		\F_0(u)\coloneqq \int_{\Rd}\int_{\Rd} K(z) f( \left| \nabla u(x)\cdot\hat{z} \right|) dz dx.
	\]
Here and in the sequel,
we set $K_h(z)\coloneqq h^{-d}K(z/h)$ with
$K\colon \R^d\to [0,+\infty)$ an even function in $L^1(\Rd)$
that has finite second moment, see \eqref{eq:K} below.
We let $f\colon [0,+\infty) \to [0,+\infty)$ be a convex function of class $C^2$
satisfying $f(0)=f'(0)=0$, and,
for $z\in\Rd \setminus\set{0}$, we put $\hat z=z/\left| z \right|$.

Functionals similar to $\F_h$ and $\F_0$
were considered by Bourgain, Brezis, and Mironescu in \cite{BBM}.
For $K$ radial and $f(t) = \left| t \right|^p$ with $p\geq1$, they established convergence as $h\searrow 0$
to a multiple of $\left\Vert{u}\right\Vert^p_{W^{1,p}(\Omega)}$ whenever $u\in W^{1,p}(\Omega)$ with $\Omega$ a smooth, bounded domain in $\Rd$.
Their result has been extended in several directions
(\cites{MaS,Da,AK,LS}, see also \cites{GM,ADM,BP}),
and, among others, we would like to spend some words
on the contributions by Ponce \cite{P}.
The author studied the case in which
$\set{K_h}$ is a suitable family of functions in $L^1(\Rd)$ that approaches the Dirac delta in $0$
and $f\colon [0,+\infty) \to [0,+\infty)$ is a generic convex function.
When $u\in L^p(\Omega)$ for some $p\geq 1$
and the boundary of $\Omega$ is compact and Lipschitz,
he showed pointwise convergence of some functionals that generalize the ones in \cite{BBM}.
The limit is a first order functional,
which is given by a variant of $\F_0$
if $K_h(z) = h^{-d}K(z/h)$ for some $K\in L^1(\Rd)$
and $u\in W^{1,p}(\Omega)$.
Further, when $\Omega$ is also bounded,
in \cite{P} $\Gamma$-convergence to the pointwise limit
with respect to the $L^1$-topology is proved too.
For the definition and the properties of $\Gamma$-convergence,
we refer to the monographs~\cites{Braides,DalMaso}.

Let now 
\begin{equation}\label{rate}
\begin{split}
\E_h (u) & \coloneqq \frac{\F_0(u)-\F_h(u)}{h^2} \\
			& = \frac{1}{h^2} \int_{\Rd}\int_{\Rd} \left[
 						K(z) f( \left|\nabla u(x)\cdot\hat{z}\right| ) - K_h(z)
 						f \left(\frac{\left| u(x+z)-u(x) \right|}{\left| z\right|}\right)
 						\right]dzdx
\end{split}
\end{equation}
be the functional which measures the rate of convergence of $\F_h$ to $\F_0$.
In this paper, under the assumptions
that the function $f$ is strongly convex
(see condition \eqref{eq:f-conv} below)
and that we restrict to functions
that vanish outside a bounded, Lipschitz set $\Omega$,
we prove that the family $\set{\E_h}$
$\Gamma$-converges, with respect to the $H^1(\Rd)$-topology, 
to the {\it second order} limit functional
\begin{equation*}
\E_0(u) \coloneqq \begin{cases}
\displaystyle{\frac{1}{24}\int_{\Rd} \int_{\Rd}
K(z) \left| z \right|^2	f''( \left| \nabla u(x) \cdot \hat z \right| ) 
\left| \nabla^2 u(x)\hat{z}\cdot \hat{z} \right|^2 dz dx} & \text{if } u \in H^2(\Rd), 
\\
+\infty & \text{otherwise}.
\end{cases}
\end{equation*}
 The uniform convexity assumption on $f$,
 which is needed for the $\Gamma$-inferior limit inequality, 
 excludes some interesting cases
 such as $f(x)=\left|x\right|^p$ with $p\ge 1$, $p\ne 2$.
 In particular, when $f(x)=\left|x\right|$ and $K$ is radially symmetric, 
the analysis is related to a geometric problem considered in \cite{MS} 
 in the context of a physical model for liquid drops with dipolar repulsion.
 We also observe that our study differs
 from a higher order $\Gamma$-limit of $\F_h$ (see~\cite{BT08}),
 which would rather correspond 
 to deal with the $\Gamma$-limit of the functionals
 \[
 \frac{\F_h-\min\F_0}{h^\alpha}\qquad \text{for some }\alpha>0\,.
 \]

As a consequence of our result (see Remark \ref{remcrit}) we also get that,
if the rate of convergence of $\F_h(u)$ to $\F_0(u)$ is fast enough,
more precisely if $\left|\E_h (u)\right| \le M$ for all $h$'s sufficiently small, then $u\in H^2(\R^d)$.  
 
 We notice that our result is reminiscent to the one obtained by
 Peletier, Planqu\'e, and R\"oger  in \cite{PR}.
 There, motivated by a model for bilayer membranes,
 the authors considered the convolution functionals
 \[
 \G_h(u) \coloneqq \int_{\Rd} f \left( K_h * u\right) dx,
 \]
 which converge to the functional $\G_0(u) = c \int_{\Rd} f \left(u\right) dx$ as $h\searrow 0$,
 with $c=c(K,d)$ a positive constant,
 and they showed that the corresponding rate functionals 
 \begin{equation}\label{rate1D}
 \frac{\G_0(u)-\G_h(u)}{h^2} =   \frac{1}{h^2}\,\int_{\Rd} \big( c\, f \left(u\right) -  f \left( K_h * u\right)\big)dx 
\end{equation} 
converge pointwise to the limit functional
 \[
 \frac 12 \int_{\Rd} \int_{\Rd} K(z) \left|z\right|^2 f''(u(x)) |\nabla u(x)\cdot \hat z|^2 dz dx\qquad
 \text{for } u \in H^1(\R^d).
 \]
 In particular, the rate functionals are uniformly bounded if and only if $u \in H^1(\R^d)$.

In the proof of our convergence result, we follow
a strategy similar to the one in \cites{G,GM}: we first consider a related $1$-dimensional problem, 
and then reduce the general case to it by a {\it slicing} procedure. More precisely, in Section \ref{sec:1D}
we study the functionals 
\[
E_h(u) \coloneqq  \frac{1}{h^2} \int_\R \left[ f(u(x)) - f\left( \fint_{x}^{x+h}u(y)dy\right) \right] dx,
\]
which are a particular case of \eqref{rate1D},
and we show their convergence  (see Theorem  \ref{stm:1D-Gconv}) to the limit energy
\[
E_0(u)\coloneqq  \frac{1}{24} \int_\R f''(u(x)) \left|u'(x)\right|^2 dx\qquad
 \text{for } u \in H^1(\R).
\]
Then, in Section \ref{sec:arb} we consider the general functionals in \eqref{rate}
and we establish the $\Gamma$-convergence to $\E_0$ (see Theorem  \ref{stm:Gconv}),
which is the main result of the present paper.
We first show the convergence for $d=1$, using the result of Section \ref{sec:1D},
and then we reduce to the $1$-dimensional case by means of a delicate slicing technique.

\noindent {\bf Acknowledgements.} MN and VP are members of INDAM-GNAMPA,
and acknowledge partial support
by the Unione Matematica Italiana and by the
University of Pisa via Project PRA 2017 {\it Problemi di ottimizzazione e di evoluzione in ambito variazionale}.
Part of this work was done during a visit of the third author to the \'Ecole Polytechnique.
VP was also supported by the Austrian Science Fund (FWF)
through the grant I4052 N32, and by BMBWF through the OeAD-WTZ project CZ04/2019.


\section{Finite difference functionals in the $1$-dimensional case}\label{sec:1D}
Let $f\colon \R \to [0,+\infty)$ be a strongly convex function of class $C^2$
such that $f(0)=f'(0)=0$.
By saying that $f$ is strongly convex, we mean that
\begin{equation}\label{eq:f-conv}
\text{there exists $\gamma>0$
	such that $f''(t)\geq\gamma$ for all $t\in\R$.}
\end{equation}

Let us fix an open interval $I\coloneqq(a,b)\subset \R$.
We introduce the closed subspace $Y\subset L^2(\R)$
defined as
\begin{equation}\label{eq:constr}
	Y \coloneqq \set{ u \in L^2(\R) : u = 0 \text{ in } \R\setminus I},
\end{equation}
and, for  $u\in Y$ and $h>0$, we define the energy
	\begin{equation}\label{eq:Eh}
		E_h(u) \coloneqq  \frac{1}{h^2} \int_\R \left[ f(u(x)) - f\left(D_h U(x)\right) \right] dx,
	\end{equation}
where 
	\begin{equation}\label{eq:U}
		U(x)\coloneqq \int_{0}^x u(y) dy
			\quad\text{and}\quad		
		D_h U(x) \coloneqq \frac{U(x+h)-U(x)}{h} = \fint_{x}^{x+h}u(y)dy.
	\end{equation}

By using some simple changes of variable
and the positivity of $f$,
one can find
	\[\begin{split}
		\int_{\R} f(u(x))dx
		& = \int_{a}^{b} f(u(x))dx
		= \int_{a-h}^{b} \fint_{x}^{x+h} f(u(y)) dy dx \\
		& = \int_{\R} \fint_{x}^{x+h} f(u(y)) dy dx,
	\end{split}\]
the integrals possibly diverging to $+\infty$.
Then, by combining the previous identity with Jensen's inequality
	\[
	\fint_x^{x+h} f(u(y)) dy \geq  f(D_h U(x)),
	\]
we see that $E_h(u)$ ranges in $[0,+\infty]$ when $u\in Y$.
	
In the present section
we compute the  $\Gamma$-limit of $\set{E_h}$
regarded as a family of functionals on $Y$
endowed with the $L^2$-topology.
Let us set
	\begin{equation} \label{eq:E0}
	E_0(u)\coloneqq  \begin{cases}
			\begin{displaystyle}
			\frac{1}{24} \int_\R f''(u(x)) \left|u'(x)\right|^2 dx
			\end{displaystyle} 		&\text{if } u\in Y \cap H^1(\R),  \\
			+\infty 						&\text{otherwise}.
	\end{cases} 
	\end{equation}
We prove the following:
	
	\begin{thm}\label{stm:1D-Gconv} 
		Let us assume that
		$f\colon \R \to [0,+\infty)$ is a function of class $C^2$
		such that $f(0)=f'(0)=0$ and \eqref{eq:f-conv} holds.
		Then, the restriction to $Y$ of the family $\set{E_h}$ $\Gamma$-converges, as $h\searrow 0$,
		to $E_0$ w.r.t. the $L^2(\R)$-topology, that is,
		for every $u\in Y$ the following properties hold:
		\begin{enumerate}
			\item \label{stm:1D-liminf}
				For any family $\set{u_h}\subset Y$ that converges to $u$ in $L^2(\R)$ we have
					\[	E_0(u)\leq \liminf_{h\searrow 0} E_h(u_h). \]
			\item \label{stm:1D-limsup}
				There exists a sequence $\set{u_h} \subset Y$ converging to $u$ in $L^2(\R)$ such that
					\[	\limsup_{h\searrow 0} E_h(u_h)\leq E_0(u).\]
		\end{enumerate}	
	\end{thm}

The $\Gamma$-upper limit is established in Proposition \ref{stm:1D-pointlim},
while Proposition \ref{stm:1D-Gliminf} takes care of the lower limit.
In turn, the latter is achieved by exploiting 
a suitable lower bound on the energy (see Lemma \ref{stm:1D-enbound})
and a compactness result (see Lemma \ref{stm:1D-cpt}),
which are a consequence of the strong convexity of $f$.



\subsection{Pointwise limit and upper bound}
We now compute the limit of $E_h(u)$, as $h\searrow 0$, for a function $u\in Y \cap C^2(\R)$.
We observe that strong convexity of $f$ is not needed
for the next proposition to hold.



	\begin{prop}\label{stm:1D-pointlim}
		Let $f\colon \R \to [0,+\infty)$ be a $C^2$ function  
		such that $f(0)=f'(0)=0$,
		and let $u\in Y\cap C^2(\R)$. Then,
		there exists a continuous, bounded, and increasing function
		$m\colon [0,+\infty) \to [0,+\infty)$ such that $m(0)=0$ and
			\begin{equation}\label{eq:1D-point-est}
				\left| E_h(u) - E_0(u) \right| \leq c\, m(h),
			\end{equation}
		where
		$c \coloneqq
				c( b-a,\lVert u \rVert_{C^2(\R)},\lVert f \rVert_{C^2([-\lVert u \rVert_{C^2(\R)},\lVert u \rVert_{C^2(\R)}])} ) >0$
		is a constant. In particular, $\lim_{h\searrow 0} E_h(u) = E_0(u)$
		and for every $u\in Y$ there exists a sequence $\set{u_h} \subset Y$
		that converges to $u$ in $L^2(\R)$ and satisfies
			\[	\limsup_{h\searrow 0} E_h(u_h)\leq E_0(u).	\]
	\end{prop}
	\begin{proof}
		Since $u\in Y\cap C^2(\R)$ and $f\in C^2(\R)$,
		it is easy to see that $h^2E_h(u)$ and $E_0(u)$ are uniformly bounded in $h$.
		Thus, there exists a constant $c_\infty>0$ such that
			\begin{equation}\label{eq:h>1}
				\left| E_h(u) - E_0(u) \right| \leq c_\infty \qquad \text{for } h>1.
			\end{equation}
		
		Next, we focus on the case $h\in(0,1]$.
		If $x\notin (a-h,b)$, then $D_h U(x) = 0$, and hence
			\[
				h^2 E_h(u) = \int_{a}^{b} \left[ f(u(x)) - f\left(D_h U(x)\right) \right]dx - \int_{a-h}^a f(D_h U(x)) dx.
			\]
		Being $u$ regular, for any $x\in(a-h,b)$
		we have the Taylor's expansion 
			\begin{equation*}
				D_h U(x)=u(x)+\frac{h}{2}u'(x)+\frac{h^2}{6}u''(x_h), \qquad\text{with $x_h\in(x,x+h)$},
			\end{equation*}
		which we rewrite as
			\begin{equation}\label{eq:Taylor-DhU}
				D_h U(x)=u(x)+h v_h(x), \qquad\text{with }\ v_h(x)\coloneqq \frac{u'(x)}{2}+\frac{h}{6}u''(x_h);
			\end{equation}
		note that $v_h$ converges uniformly to $u'/2$ as $h\searrow 0$.
		
		Plugging \eqref{eq:Taylor-DhU} into the definition of $E_h$, we get
			\[\begin{split}
				h^2 E_h(u) = & - \int_a^b \left[f\big( u(x) + h v_h(x) \big) - f(u(x)) \right]dx
										- \int_{a-h}^a f\left( \frac{h^2}{6} u''(x_h) \right) dx \\
								= & -h\int_a^b f'(u(x))v_h(x) dx - \frac{h^2}{2}\int_a^b f''(w_h(x))v_h(x)^2 dx \\
									& - \int_{a-h}^a f\left( \frac{h^2}{6} u''(x_h) \right) dx,
			\end{split}\]
		where $w_h$ fulfils $w_h(x)\in(u(x),u(x)+hv_h(x_h))$ for all $x\in(a,b)$.
		
		In view of the regularity of $f$ and $u$,
		we can utilize the Mean Value Theorem to obtain
			\[
				\left| \int_{a-h}^a f\left( \frac{h^2}{6} u''(x_h) \right) dx \right| \leq c_1 h^5
			\]
		for a constant $c_1>0$
		that depends only on $N\coloneqq \lVert u \rVert_{C^2(\R)}$
		and on $\lVert f'' \rVert_{L^\infty( [-N,N] )}$.
		Moreover, recalling the definition of $v_h$, we have
			\[
				\int_a^b f'(u(x))v_h(x) dx
					=  \frac{h}{6}\int_a^b f'(u(x)) u''(x_h) dx,
			\]
		and therefore
			\begin{equation}\label{eq:Eh-E0}
			\begin{split}
				\left| E_h(u) - E_0(u) \right| 
					\leq & \frac{1}{6}\left| - \int_a^b f'(u(x)) u''(x_h) dx
									- \int_a^b f''(u(x)) u'(x)^2 dx \right|
					\\ & + \frac{1}{2} \left| \frac{1}{4} \int_a^b f''(u(x)) u'(x)^2 dx
									- \int_a^b f''(w_h(x))v_h(x)^2 dx \right| 
							+ c_1 h^5.
			\end{split}		
			\end{equation}
		Since $u\in Y\cap C^2(\R)$,
		$u''$ admits a uniform modulus of continuity $m_{u''}\colon [0,+\infty) \to [0,\infty)$.
		An integration by parts gives that
			\[\begin{split}
				\left| -\int_a^b f'(u(x)) u''(x_h) dx - \int_a^b f''(u(x)) u'(x)^2 dx \right| 
					\leq &  \int_a^b \left| f'(u(x)) \right|  \left| u''(x) - u''(x_h) \right| dx \\
					\leq & \ c_2 m_{u''}(h),
			\end{split}\]
		where $c_2 \coloneqq (b-a) \lVert f' \rVert_{L^\infty( [-N, N] )}$.
		
		In a similar manner,
		denoting by $m_{f''}$ the modulus of continuity of the restriction of $f''$ to the interval $[-N,N]$,
		we also find
 			\begin{multline*}
				\left| \frac{1}{4} \int_a^b f''(u(x)) u'(x)^2 dx - \int_a^b f''(w_h(x))v_h(x)^2 dx \right| 
					\\ \leq  \int_a^b \left| f''(u(x)) \right|  \left| \frac{1}{4} u'(x)^2 -  v_h(x)^2\right| dx 
						+   \int_a^b \left| f''(u(x)) - f''(w_h(x)) \right|  v_h(x)^2 dx
					\\ \leq c_3( h + m_{f''}(h) ),
		\end{multline*}
		with $c_3$ depending on $b-a$, $N$, and $\lVert f'' \rVert_{L^\infty( [-N, N] )}$.
		
		By combining \eqref{eq:Eh-E0} with the inequalities above,
		we obtain
			\begin{equation}\label{eq:h<1}
				\left| E_h(u) - E_0(u) \right| \leq c_0 \big( m_{u''}(h) + m_{f''}(h) +h + h^5 \big)
						\qquad \text{for } h\in(0,1],
			\end{equation}
		for a suitable constant $c_0>0$.		
		At this stage, \eqref{eq:1D-point-est} follows by combining \eqref{eq:h>1} and \eqref{eq:h<1}.
		
		As for the existence of a family that fulfils the upper limit inequality,
		we apply a standard density argument that we sketch in the following lines.		
		Let $u\in Y$. If $u\notin H^1(\R)$,
		the inequality holds trivially;
		otherwise, by rescaling the domain and mollifying,
		we construct a sequence of smooth functions $\set{u_\ell} \subset Y$
		that converges to $u$ both uniformly and in $H^1(\R)$.
		Then, since $f''$ is a continuous function,
		we get
		$\lim_{\ell \nearrow +\infty} E_0(u_\ell) = E_0(u)$.
		Besides, we know that
		$\lim_{h \searrow 0} E_h(u_\ell) = E_0(u_\ell)$
		for any $\ell\in \mathbb{N}$, because $u_\ell$ is smooth.
		We conclude that
		there exists a subsequence $\set{h_\ell}$ such that
			\[
			\lim_{\ell \nearrow +\infty} E_{h_\ell}(u_\ell) = E_0(u).
			\]
	\end{proof}


\begin{rmk}\label{remscala1}
Notice that, as a consequence of Proposition \ref{stm:1D-pointlim}, the $\Gamma$-limit of 
the rate functionals 
\[
h E_h(u) = \frac{1}{h} \int_\R \left[ f(u(x)) - f\left(D_h U(x)\right) \right] dx
\]
is equal to zero.
\end{rmk}


\subsection{Lower bound in the strongly convex case}
In view of Proposition \ref{stm:1D-pointlim},
to accomplish the proof of Theorem \ref{stm:1D-Gconv},
it only  remains to establish statement \eqref{stm:1D-liminf},
that is, for any $u\in Y$ and for any family $\set{u_h}\subset Y$ converging to $u$ in $L^2(\R)$
it holds
	\[	E_0(u)\leq \liminf_{h\searrow 0} E_h(u_h). 	\]

In the current subsection we utilize the strong convexity of  the function $f$.
We exploit this hypothesis to provide a lower bound on the energy $E_h$, by means of which
we prove that
sequences with equibounded energy are relatively compact w.r.t. the $L^2$-topology.
	
	\begin{lemma}[Lower bound on the energy]\label{stm:1D-enbound}
	Let us assume that
	$f\colon \R \to [0,+\infty)$ is a function of class $C^2$
	such that $f(0)=f'(0)=0$ and \eqref{eq:f-conv} is fulfilled.
	Then, for any $u\in Y$, it holds
		\begin{equation}\label{eq:1D-lbound}
			E_h(u)\geq \sup_{\phi \in C^\infty_c\left(\R^2\right)}
								\left\{ \int_\R\fint_x^{x+h}
											\left(
												\frac{u(y) - D_hU(x)}{h} \phi(x,y) - \frac{\phi(x,y)^2}{4\lambda_h(x,y)}
											\right) dy dx
								\right\},
		\end{equation}
	with		
		\begin{equation} \label{eq:lambda}
			\lambda_h(x,y)\coloneqq \int_0^1 (1-\theta)f''\big((1-\theta)D_hU(x)+\theta u(y)\big)d\theta.
		\end{equation}
	Moreover, 
	\begin{equation}\label{eq:1D-lbound2}
		E_h(u) \geq \frac{\gamma}{4} \int_\R\int_{-h}^{h}
							J_h(r)\left(\frac{u(y+r) - u(y)}{h}\right)^2 drdy,
		\end{equation}
	where
		\begin{equation*}
		J(r)\coloneqq (1-|r|)^+
			\quad\text{and}\quad
		J_h(r)\coloneqq \frac{1}{h}J\left(\frac{r}{h}\right).
		\end{equation*}
	\end{lemma}

	\begin{proof}
	For a given $h>0$, let us consider $u\in Y$ such that $E_h(u)$ is finite.
	We write
		\[\begin{gathered}
			E_h(u) = \frac{1}{h^2} \int_\R e_h(x) dx,
			\quad\text{where } e_h(x) \coloneqq \fint_x^{x+h} [f(u(y))-f(D_hU(x))] dy.
		\end{gathered}\]
	Thanks to the identity
		\[f(s)-f(t) = f'(t)(s-t)+(s-t)^2 \int_0^1 (1-\theta) f''((1-\theta)t+\theta s))d\theta, \]
	we find
		\begin{equation}\label{eq:Ch}
		e_h(x) = \fint_x^{x+h} \lambda_h(x,y) \left(u(y) - D_hU(x)\right)^2 dy,
		\end{equation}
	where $\lambda_h(x,y)$ is as in \eqref{eq:lambda}.
	Observe that for any $\phi \in C^\infty_c\left(\R^2\right)$
	we have the pointwise inequality
		\[
		\lambda_h(x,y) \left(u(y) - D_hU(x)\right)^2 \geq \frac{u(y) - D_hU(x)}{h} \phi(x,y) - \frac{\phi(x,y)^2}{4\lambda_h(x,y)},
		\]
	from which we infer \eqref{eq:1D-lbound}.
	
	The strong convexity of $f$ grants that
	$\lambda_h(x,y) \ge \gamma / 2$ for all $(x,y)\in \R^2$ and $h>0$,
	thus we also deduce that
		\[
			e_h(x)\geq \frac{\gamma}{2}\fint_x^{x+h}\left(u(y) - D_hU(x)\right)^2.
		\]
	Hence, we get
		\begin{equation*}\begin{split}
			E_h(u)\geq &\frac{\gamma}{2} \int_\R\fint_x^{x+h}\left(\frac{u(y) - D_hU(x)}{h}\right)^2dydx \\
						\geq & \frac{\gamma}{4} \int_\R\fint_x^{x+h}\fint_x^{x+h}\left(\frac{u(z) - u(y)}{h}\right)^2 dzdydx,
		\end{split}
		\end{equation*}
	where the last inequality follows from the identity
		\[ 
		\int |\phi(y)|^2d\mu(y) = \left|\int \phi(y) d\mu(y)\right|^2 + \frac{1}{2}\int\int |\phi(z)-\phi(y)|^2d\mu(z)d\mu(y),
		\]
	which holds whenever $\mu$ is a probability measure and $\phi\in L^2(\mu)$.
	By Fubini's Theorem and neglecting contributions near the boundary,
	we find the lower bound on the energy:
		\begin{equation*}
		\begin{split}
			E_h(u)\geq &\frac{\gamma}{4} \int_\R\fint_{y-h}^y\fint_x^{x+h}
															\left(\frac{u(z) - u(y)}{h}\right)^2 dzdxdy \\
					= &\frac{\gamma}{4h} \int_\R\int_{y-h}^{y+h}
															\left(1-\frac{\left|z-y\right|}{h}\right)\left(\frac{u(z) - u(y)}{h}\right)^2 dzdy.
		\end{split}
		\end{equation*}
	The conclusion \eqref{eq:1D-lbound2} is now achieved by the change of variables $r=z-y$.
		\end{proof}
	
	\begin{rmk}\label{rmk:f''bounded}
		Let $u\in Y \cap H^1(\R)$.
		If along with the previous assumptions
		we also require that $f''$ is bounded above by a constant $c>0$,
		then the family $\set{E_h(u)}$ is uniformly bounded.
		Indeed, it follows from \eqref{eq:Ch} and the definition of $\lambda_h$
		that
		\begin{equation*}
			E_h(u) \leq \frac{c}{2 h^2} \int_{\R} \fint_x^{x+h} \left(u(y) - D_hU(x)\right)^2 dy dx.
		\end{equation*}
		Then, since $u\in H^1(\R)$,
		when $y\in (x,x+h)$ it holds
			\[
				\left| u(y) - D_hU(x) \right|^2
				 \leq h \int_{x}^{x+h} \left| u'(z)\right|^2 dz
					= h^2 \int_{0}^{1} \left| u'(x + hz)\right|^2 dz,
			\]
		and we derive the estimate
			\begin{equation}\label{eq:Ehbounded}
			E_h(u) \leq \frac{c}{2} \int_{\R} \left| u'(x )\right|^2 dx.
			\end{equation}
	\end{rmk}
		
	\begin{lemma}[Compactness]\label{stm:1D-cpt}
	Let the function $f$ be as in Lemma \ref{stm:1D-enbound} and
	let $\set{u_h}\subset Y$ be a sequence of functions such that
	$E_h(u_h)\leq M $ for some $M\geq 0$.
	Then, there exist a subsequence $\set{u_{h_\ell}}$ and a function $u\in Y \cap H^1(\R)$
	such that $u_{h_\ell}\to u$ in $L^2(\R)$.		
	\end{lemma}

	\begin{proof}	
	We adapt the strategy of \cite[Theorem 3.1]{AB}.
	
	By Lemma \ref{stm:1D-enbound}, we infer that
		\begin{equation}\label{eq:1D-lbound-bis}
		\frac{\gamma}{4} \int_\R\int_{-h}^{h}
			J_h(r)\left(\frac{ u_h(y+r) -  u_h(y)}{h}\right)^2 dr dy \leq M.
		\end{equation}
	Observe that $J_h(r)dr$ is a probability measure on $[-h,h]$.
	
	We now introduce the mollified functions $v_h \coloneqq \rho_h \ast  u_h$,
	where $\set{\rho_h}$ is the family 
		\[	\rho_h(r)\coloneqq \frac{1}{ch}\rho\left(\frac{r}{h}\right), \qquad \text{with }c\coloneqq \int_\R\rho(r)dr.	\]
	Here, $\rho\in C^{\infty}_c(\R)$ is an even kernel,
	and it is chosen in such a way that its support is contained in $[-1,1]$,
		\[	0\leq \rho \leq J, \quad\text{and}\quad \left|\rho'\right| \leq J.\]
	Note that, for all $h>0$, $v_h\colon \R\to\R$ is a smooth function
	whose support is a subset of $(a-h,b+h)$.
	Moreover, the family of derivatives $\set{v'_h}_{h\in(0,1)}$ is uniformly bounded in $L^2(\R)$;
	indeed, since $\int_\R \rho'(r)dr =0$, it holds
		\[\begin{split}
			\int_\R \left| v_h'(y) \right|^2 dy 
				= & \int_\R \left| \int_{-h}^{h} \rho'_h(r)[ u_h(y+r) - u_h(y)]dr \right|^2 dy \\
				\le &  \int_\R\left(\int_{-h}^{h} \left|\rho'_h(r)\right|\left| u_h(y+r) - u_h(y) \right| dr \right)^2 dy \\
				\le &  \frac{1}{c^2}\int_\R 
							\left(\int_{-h}^{h} J_h(r) \left| \frac{ u_h(y+r) - u_h(y)}{h}\right| dr \right)^2 dy \\
				\le &  \frac{1}{c^2}\int_\R \int_{-h}^{h} J_h(r) \left|\frac{ u_h(y+r) -  u_h(y)}{h}\right|^2 dr dy,
		\end{split}\]
	and thus
		\begin{equation}\label{eq:vh-H^1}	
		\int_\R \left|v_h'(y)\right|^2 dy \leq \frac{4M}{c^2\gamma}.
		\end{equation}
	For all $h\in(0,1)$, let $\tilde v_h$ be the restriction of $v_h$ to the interval $(a-1,b+1)$.
	By Poincar\'e inequality, \eqref{eq:vh-H^1} entails boundedness in $H^1_0((a-1,b+1))$
	of the family $\set{\tilde v_h}_{h\in(0,1)}$,
	and, in view of Sobolev's Embedding Theorem, this grants in turn that
	there exists a subsequence $\set{\tilde v_{h_\ell}}$
	uniformly converging to some $ \tilde u\in H^1_0([a-1,b+1])$.
	Since each $\tilde v_{h_\ell}$ is supported in $(a-h_\ell,b+h_\ell)$,
	we see that $\tilde u\in H^1_0(\bar I)$; therefore, if we set
		\[
			u(x)\coloneqq\begin{cases}
					\tilde u(x) 	&\text{if } x\in \bar I, \\
					0					& \text{otherwise},
					\end{cases}
		\]
	we deduce that $\set{v_{h_\ell}}$ converges uniformly to $u \in Y\cap H^1(\R)$.	
	
	Lastly, to achieve the conclusion,
	we provide a bound on the $L^2$-distance between $u_h$ and $v_h$.
	Similarly to the previous computations, we have 
		\[\begin{split}
			\int_\R \left|v_h(y)- u_h(y)\right|^2 dy 
						= & \int_\R\left|\int_{-h}^{h} \rho_h(r)[ u_h(y+r) - u_h(y)]dr \right|^2 dy \\
						\leq & \int_\R \int_{-h}^{h} \rho_h(r)\left| u_h(y+r) - u_h(y)\right|^2 dr dy \\
						\leq & \frac{1}{c}\int_\R \int_{-h}^{h} J_h(r) \left| u_h(y+r) - u_h(y) \right|^2 dr dy,
		\end{split}	\]
	and, by \eqref{eq:1D-lbound-bis}, we get
		\begin{equation}\label{eq:vh-L2}
			\int_\R |v_h(y) - u_h(y)|^2 dy  \leq \frac{4M}{c\gamma} h^2.
		\end{equation}
	Since there exists a subsequence $\set{v_{h_\ell}}$
	uniformly converging to a function $u \in Y\cap H^1(\R)$,
	\eqref{eq:vh-L2} gives the conclusion.
\end{proof}

Now we can prove statement \eqref{stm:1D-liminf} of Theorem \ref{stm:1D-Gconv}.

	\begin{prop}\label{stm:1D-Gliminf}
		Let the function $f$ be as in Lemma \ref{stm:1D-enbound}.
		Then, for any $u\in Y$ and for any family $\set{u_h}\subset Y$ that converges to $u$ in $L^2(\R)$,
		it holds
			\begin{equation}\label{eq:1D-Gliminf}
				E_0(u)\leq \liminf_{h\searrow 0} E_h(u_h).
			\end{equation}
	\end{prop}
	
	\begin{proof}
		Fix $u,u_h\in Y$ in such a way that $u_h\to u$ in $L^2(\R)$.
		We can suppose that the inferior limit in \eqref{eq:1D-Gliminf} is finite,
		otherwise the conclusion holds trivially.
		Consequently, up to extracting a subsequence, which we do not relabel,
		there exists $\lim_{h\searrow 0} E_h(u_h)$ and it is finite.
		In particular, there exists $M\geq 0$ such that $E_h(u_h)\leq M$ for all $h>0$,
		and, by Lemma \ref{stm:1D-cpt}, this yields that $u\in Y \cap H^1(\R)$.
		
		We use formula \eqref{eq:1D-lbound} for each $u_h$,
		choosing, for $(x,y) \in \R^2$, 
			\[
				\phi(x,y) = \psi\left(x,\frac{y-x}{h}\right),
						\qquad \text{with } \psi\in C_c^\infty(\R^2).
			\]
		We get
			\begin{equation}\label{eq:psi-estimate}
				\begin{split}
				E_h(u_h) \geq	
						& \int_\R\fint_x^{x+h}\frac{u_h(y) - \fint_{x}^{x+h} u_h}{h} \psi\left(x,\frac{y-x}{h}\right) dy dx \\
						& - \frac{1}{4} \int_\R\fint_x^{x+h}\frac{\psi\left(x,\tfrac{y-x}{h}\right)^2}{\lambda_h(x,y)} dy dx,
				\end{split}
			\end{equation}
		where, coherently with \eqref{eq:lambda},
			\begin{equation*}
			\lambda_h(x,y) \coloneqq \int_0^1 (1-\theta)f''\left((1-\theta)\fint_{x}^{x+h}u_h(z)dz+\theta u_h(y)\right)d\theta
			 \geq \frac{\gamma}{2}.
			\end{equation*}
		
		Let us focus on the first quantity on the right-hand side of \eqref{eq:psi-estimate}.
		We have
			\[\begin{split}
				\frac{1}{h}\int_\R\fint_x^{x+h} &
							\left(\fint_{x}^{x+h} u_h(z) dx \right)\psi\left(x,\frac{y-x}{h}\right) dy dx \\
				= & \frac{1}{h^3}\int_\R\int_x^{x+h}\int_x^{x+h} u_h(z) \psi\left(x,\frac{y-x}{h}\right) dy dz dx \\
				= 
					& \frac{1}{h^3}\int_\R\int_{z-h}^{z} \int_x^{x+h}u_h(z) \psi\left(x,\frac{y-x}{h}\right) dy dx dz,
			\end{split}\]
		and, by similar computations, we obtain
			\begin{equation}\label{eq:psi}
				\begin{split}
				\int_\R\fint_x^{x+h} & \frac{u_h(y) - \fint_{x}^{x+h} u_h}{h} \psi\left(x,\frac{y-x}{h}\right) dy dx \\
				 	= 
				 		& \frac{1}{h}\int_\R\fint_{y-h}^{y} \fint_x^{x+h}
				 					u_h(y)\left[\psi\left(x,\frac{y-x}{h}\right) - \psi\left(x,\frac{z-x}{h}\right)\right]
				 					dz dx dy.
				\end{split}
			\end{equation}
		By a simple change of variable, we get
			\[\begin{gathered}
				\begin{split}
					\fint_{y-h}^{y} \fint_x^{x+h} \psi\left(x,\frac{y-x}{h}\right) dz dx
						= & \int_0^1\int_0^1\psi(y - h r,r) dq dr,
				\end{split} \\
				\begin{split}
					\fint_{y-h}^{y} \fint_x^{x+h} \psi\left(x,\frac{z-x}{h}\right) dz dx
					= & \fint_{y-h}^{y} \int_0^1 \psi(x,r) dr dx \\
					= & \int_0^1 \int_0^1 \psi(y-hq,r) dq dr,
				\end{split}
			\end{gathered}\]
		hence
			\[\begin{split}
				\frac{1}{h} \fint_{y-h}^{y} \fint_x^{x+h} 
						& \left[\psi\left(x,\frac{y-x}{h}\right) - \psi\left(x,\frac{z-x}{h}\right)\right] dz dx \\
					= & \int_0^1 \int_0^1 \frac{ \psi(y - hr, r) -  \psi(y - hq, r) }{h} dq dr \\
					= & -\int_0^1 \int_0^1 \int_q^r \partial_1 \psi(y - hs, r) ds dq dr \\
					= & -\int_0^1 \int_0^1 (r-q) \fint_q^r \partial_1 \psi(y-h s,r) ds dq dr.
			\end{split}\]
		Being $\psi$ smooth, we have that
		$\partial_1\psi(y - hs,r) = \partial_1\psi(y,r)+O(h)$ as $h\searrow 0$, uniformly for $s\in [0,1]$.
		Consequently,
			\[
				\frac{1}{h} \fint_{y-h}^{y} \fint_x^{x+h} 
						 \left[\psi\left(x,\frac{y-x}{h}\right) - \psi\left(x,\frac{z-x}{h}\right)\right] dz dx 
						 = -  \int_0^1 \left(r-\frac{1}{2}\right) \partial_1\psi(y,r) dr + O(h).
			\]
		Plugging this equality in \eqref{eq:psi} yields
			\[
				\int_\R\fint_x^{x+h} \frac{u_h(y) - \fint_{x}^{x+h} u_h}{h} \psi\left(x,\frac{y-x}{h}\right) dy dx
							=  - \int_\R u_h(y)
									\int_0^1 \left(r-\frac{1}{2}\right) \partial_1\psi(y,r) dr + O(h).
			\]
		It is possible to take the limit $h\searrow 0$ in the previous formula,
		since $u_h \to u$ in $L^2(\R)$. We then get
			\begin{equation}\label{eq:1st}
				\lim_{h\searrow 0} 
						\int_\R\fint_x^{x+h} \frac{u_h(y) - \fint_{x}^{x+h} u_h}{h} \psi\left(x,\frac{y-x}{h}\right) dy dx 
						=- \int_\R \int_0^1 u(y) (r-\tfrac{1}{2}) \partial_1\psi(y,r) dr dy.
			\end{equation}
			
		Now, we turn to the second addendum on the right-hand side of \eqref{eq:psi-estimate}.
		By Fubini's Theorem and a change of variables, we have
			\begin{equation*}
				\begin{split}
					\int_\R\fint_x^{x+h} \frac{\psi(x,\frac{y-x}{h})^2}{\lambda_h(x,y)} dy dx 
							=
								\int_\R \int_{0}^{1} \frac{\psi\left(y - hr, r\right)^2}{\lambda_h(y - hr, y)} dr dy.
				\end{split}
			\end{equation*}
		The function $\psi$ has compact support and $\lambda_h\geq \gamma /2$ for all $h>0$,
		therefore we can apply Lebesgue's Convergence Theorem
		to let $h\searrow 0$ in the previous expression, and
		we get
			\begin{multline*}
				\lim_{h\searrow 0} \int_\R \int_{0}^{1}
							\frac{\psi\left(y - hr, r\right)^2}{\int_0^1 (1-\theta) f''\left((1-\theta)\fint_{y - hr}^{y +(1-r)h}u_h(z)dz+\theta u_h(y)\right)d\theta} dr dy \\				 
					= \int_\R \int_0^1 \frac{\psi(y,r)^2}{\int_0^1 (1-\theta) f''(u(y))d\theta} dr dy,
			\end{multline*}
		thus
			\begin{equation}\label{eq:2nd}
				\lim_{h\searrow 0} \int_\R\fint_x^{x+h} \frac{\psi(x,\frac{y-x}{h})^2}{\lambda_h(x,y)} dy dx = 
					2\int_\R \int_0^1 \frac{\psi(y,r)^2}{ f''(u(y))} dr dy.
			\end{equation}		
		Summing up, by \eqref{eq:1st} and \eqref{eq:2nd},
		we deduce
			\begin{equation}\label{eq:liminf-Eh}
			\liminf_{h \searrow 0} E_h(u_h) \geq 
						- \int_\R \int_0^1 \left[
								u(y) \left(r-\frac{1}{2}\right) \partial_1\psi(y,r) dr dy
									+ \frac{1}{2} \int_\R \int_0^1 \frac{\psi(y,r)^2}{ f''(u(y))}
							\right] dr dy,
			\end{equation}
		for all $\psi \in C^\infty_c (\R^2)$.
		
		We can reach the conclusion from the last inequality by a suitable choice of the test function $\psi$.
		To see this, we let $\eta\in C^\infty_c(\R)$ and we choose a standard sequence of mollifiers $\set{\rho_k}$.
		We then set
			\[
				\psi(x,y) = \psi_k(x,y) \coloneqq \eta(x) \left(\zeta_k(y)- \frac{1}{2}\right),
				\quad\text{with } \zeta_k(y)\coloneqq \int_{\R} \rho_k(z-y) z dz,
			\]
		so that \eqref{eq:liminf-Eh} reads
			\[
			\begin{split}
				\liminf_{h \searrow 0} E_h(u_h) \geq &
				- \int_0^1 \left(r-\frac{1}{2}\right)\left(\zeta_k(r)- \frac{1}{2}\right) dr
					\int_\R  u(y) \eta'(y) dy \\
				& - \frac{1}{2} \int_0^1 \left(\zeta_k(r)- \frac{1}{2}\right)^2 dr
				 \int_\R \frac{\eta(y)^2}{ f''(u(y))} dy.
			\end{split}
			\]
		Because of the identity $\int_0^1 (r -1/2)^2 dr = 1/12$, letting $k\to +\infty$ yields
			\[
			\begin{split}
			\liminf_{h \searrow 0} E_h(u_h) \geq 
						& - \frac{1}{12} \left[ \int_\R u(y) \eta'(y) dy
								+ \frac{1}{2} \int_\R \frac{\eta(y)^2}{ f''(u(y))} dy \right] \\
					= &  \frac{1}{12} \left[ \int_\R u'(y) \eta(y) dy
					- \frac{1}{2} \int_\R \frac{\eta(y)^2}{ f''(u(y))} dy \right],
			\end{split}
			\]
		where $u'\in L^2(\R)$ is the distributional derivative of $u$,
		which exists since $u\in H^1(\R)$.		
		Recall that, in the previous formula, the test function $\eta$ is arbitrary,
		thus, to recover \eqref{eq:1D-Gliminf}, it suffices to take the supremum w.r.t. $\eta \in C^\infty_c(\R)$.
	\end{proof}

        \comment{}


\section{$\Gamma$-limit in arbitrary dimension}\label{sec:arb}
Let us fix the assumptions and the notation
that we use in the current section.  
We consider an open, bounded set $\Omega\subset\Rd$ with Lipschitz boundary
and a measurable function $K\colon \Rd \to [0,+\infty)$ such that
	\begin{equation}\label{eq:K}
		\int_{\Rd} K(z)\left( 1 + \left| z \right|^2 \right) dz < +\infty.
	\end{equation}
We require that $K(z) = K(-z)$ for a.e. $z\in \Rd$
and that the support of $K$
contains a sufficiently large annulus centered at the origin.
More precisely, let us set
	\begin{equation}\label{eq:sigma}
		\sigma_d \coloneqq
			\begin{cases}
			1												& \text{when } d=2, \\
			\displaystyle{\frac{d-2}{d-1}}	& \text{when } d>2;
			\end{cases}
	\end{equation}
we suppose that there exist $r_0\geq 0$ and $r_1>0$ such that
$r_0 < \sigma_d r_1$ and
	\begin{equation}\label{eq:Kpos}
	\mathrm{ess} \inf \set{ K(z) : z \in B(0,r_1)\setminus B(0,r_0) } > 0.
	\end{equation}
The simplest case for which \eqref{eq:Kpos} holds is  
	when there exists $k>0$ such that $K(z) \geq k$ for all $z\in B(0,r_1)$.
Finally, we let $f\colon [0,+\infty) \to [0,+\infty)$ be a $C^2$ function
such that $f(0) = f'(0) = 0$ and the strong convexity condition \eqref{eq:f-conv} is satisfied.

For  $u\in H^1(\Rd)$, we define the functionals
		\[\begin{gathered}
			\F_h(u)\coloneqq \int_{\Rd}\int_{\Rd}
			K_h(y-x) f \left(\frac{\left| u(y)-u(x) \right|}{\left| y-x \right|}\right) dydx, \\
			\F_0(u)\coloneqq \int_{\Rd}\int_{\Rd} K(z) f( |\nabla u(x)\cdot\hat{z}| ) dz dx,
		\end{gathered}\]
where $\hat{z}\coloneqq z \slash \left| z \right|$ for $z\neq 0$ and $K_h(z)\coloneqq h^{-d}K(z/h)$.		

\begin{rmk}
	By appealing to the results in \cite{P},
	one can show that
	$\F_h(u)$ tends to $\F_0(u)$ as $h\searrow 0$
	when $u$ is smooth enough and vanishes in $\Rd \setminus \Omega$,
	and also that $\F_0$ is the $\Gamma$-limit of the family $\set{\F_h}$.
	Indeed, if $u=0$ a.e. in $\Rd \setminus \Omega$, we have
	\[\begin{split}
		\F_h(u) & \coloneqq
			\int_{\Omega}\int_{\Omega}
				K_h(y-x) f \left(\frac{\left| u(y)-u(x) \right|}{\left| y-x \right|}\right) dydx \\
			\qquad & + 2 \int_{\Omega}\int_{\Rd \setminus \Omega}
				K_h(y-x) f \left(\frac{\left| u(x) \right|}{\left| y-x \right|}\right) dydx.
	\end{split}\]
By \cite{P}, we know that
the first addendum on the right-hand side
converges and $\Gamma$-converges to $\F_0$.
It is clear that this is also the $\Gamma$-inferior limit of $\set{\F_h}$,
because the term
	\[
	\tilde \F_h(u) \coloneqq 2 \int_{\Omega}\int_{\Rd \setminus \Omega}
	K_h(y-x) f \left(\frac{\left| u(x) \right|}{\left| y-x \right|}\right) dydx
	\]
is positive and may be dropped.
As for the pointwise limit,
we pick a function $u\in C^1(\Rd)$
that equals $0$ in $\Rd\setminus\Omega$,
and we observe that
the quotient $\left| u(x) \right| / \left| y - x \right|$ is bounded above by $\left\Vert u \right\Vert_{C^1(\Omega)}$.
It follows that
	\[
		\tilde F_h(u) \leq c \int_{\Omega}\int_{\Rd \setminus \Omega}
		K_h(y-x) dy dx,
	\]
with $c\coloneqq 2\left\Vert f \right\Vert_{L^\infty ([0,\left\Vert u \right\Vert_{C^1(\Rd)}])}$,
whence $\lim_{h\searrow 0} \tilde F_h(u) = 0$
(recall that $K\in L^1(\Rd)$).
\end{rmk}

Analogously to the $1$-dimensional case, we define
	\[
		\E_h (u)\coloneqq \frac{\F_0(u)-\F_h(u)}{h^2}
	\]
and we study the asymptotics of this family as $h \searrow 0$.
Notice that the functionals $\E_h$ are positive
(see Lemma \ref{stm:slicing} below).
Let us set
	\begin{equation}\label{eqX}
			X \coloneqq \set{ u\in H^1(\Rd) : u=0 \text{ a.e. in } \Rd \setminus \Omega }
	\end{equation}
and
	\begin{equation}\label{eq:E0-d}
		\E_0(u) \coloneqq
			\begin{cases}
			\displaystyle{
				\frac{1}{24}\int_{\Rd} \int_{\Rd}
				K(z) \left| z \right|^2	f''( \left| \nabla u(x) \cdot \hat z \right| ) 
				\left| \nabla^2 u(x)\hat{z}\cdot \hat{z} \right|^2
				dz dx}  \hfill \\
				\hfill \text{if } u \in X\cap H^2(\Rd),
			\\
			+\infty \hfill \text{otherwise}.					
			\end{cases}
	\end{equation}
We observe that
if $u\in X\cap H^2(\Rd)\cap W^{1,\infty}(\Rd)$
or if $f''$ has quadratic growth at infinity and $u\in X\cap H^2(\Rd)$,
then $\E_0(u)$ is finite.
	
	\begin{rmk}[Radial case]
		When $K$ is radial, that is
		$K(z) = \bar K( \left| z \right| )$ for some $\bar K\colon [0,+\infty) \to [0,+\infty)$,
		we have
			\[\begin{gathered}
				\F_0 (u) = \lVert K \rVert_{L^1(\Rd)} \int_{\Rd} \fint_{\mathbb{S}^{d-1}}
										f( \left| \nabla u(x)\cdot e \right| ) d\H^{d-1}(e) dx, \\
				\E_0 (u) = \frac{1}{24} \left( \int_{\Rd} K(z) \left| z \right|^2 dz \right)
								\int_{\Rd} \fint_{\mathbb{S}^{d-1}}
					 				f''( \left| \nabla u(x) \cdot e \right| ) 
									\left| \nabla^2 u(x)e\cdot e \right|^2
								d\H^{d-1}(e) dx.
			\end{gathered}\]
	\end{rmk}
	
This Section is devoted to the proof of the following:
	
	\begin{thm}\label{stm:Gconv}
		Let $\Omega$, $K$, and $f$ satisfy the assumptions stated at the beginning of the current section.
		Then, there hold:
		\begin{enumerate}
		\item\label{stm:cptgen} For any family $\set{u_h}\subset X$ such that $\E_h(u_h)\le M$ for some $M>0$,
		there exists a subsequence  $\set{u_{h_\ell}}$ and a function $u\in X \cap H^2(\R^d)$
		such that $\nabla u_{h_\ell}\to \nabla u$ in $L^2(\Rd)$.
			\item\label{stm:Gliminf} For any family $\set{u_h}\subset X$
				that converges to $u\in X$ in $H^1(\Rd)$
					\[	\E_0(u)\leq \liminf_{h\searrow 0} \E_h(u_h). \]
		\end{enumerate}
		\begin{enumerate}[label=(3\alph*)]
			\item\label{stm:Glimsup-a} For any $u\in X\cap W^{1,\infty}(\Rd)$
			there exists a family $\set{u_h}\subset X$
			that converges to $u$ in $H^1(\Rd)$ with the property that
			\[	\limsup_{h\searrow 0} \E_h(u_h) \leq \E_0(u). \]
			\item\label{stm:Glimsup-b} If $f''$ is bounded,
			for any $u\in X$ there exists a family $\set{u_h}\subset X$
			that converges to $u$ in $H^1(\Rd)$ with the property that
			\[	\limsup_{h\searrow 0} \E_h(u_h) \leq \E_0(u). \]
		\end{enumerate}
	\end{thm}

Statements \ref{stm:Gliminf}, \ref{stm:Glimsup-a}, and \ref{stm:Glimsup-b} amounts to saying that 
$\E_0$ is the $\Gamma$-limit of $\set{\E_h}$
with respect to the $H^1(\Rd)$-convergence
if either we restrict to functions in $X	\cap W^{1,\infty}(\Rd)$  or
$f''$ is bounded.

\subsection{Slicing}
When the dimension is $1$,
by virtue of the analysis in Section \ref{sec:1D},
it is not difficult to derive the $\Gamma$-convergence of the functionals $\E_h$.
	
	\begin{cor}\label{stm:1D-cor} 
		Let $K\colon \R \to [0,+\infty)$ be an even function such that \eqref{eq:K} holds.
		For $h>0$ and $u\in H^1(\R)$, we define the family
			\[
				\E_h(u) \coloneqq \frac{1}{h^2}\int_{\R} \int_{\R}
							K_h(z) \left[
									f( \left| u'(x)\right| ) -  f \left(\left|\frac{u(x+z)-u(x)}{z}\right|\right)
								\right] dzdx.
			\]			
		We also let $\Omega = (a,b)$ be an open interval, $f\colon [0,+\infty) \to [0,+\infty)$ be a $C^2$ function satisfying
		$f(0) = f'(0) = 0$ and 
		$f''(t)\geq\gamma$ with $\gamma>0$ for all $t\in\R$,
		and $X\subset H^1(\R)$ be as in \eqref{eqX}.
		Then, the restrictions of the functionals $\E_h$ to $X$ $\Gamma$-converge
		w.r.t. the $H^1(\R)$-topology to
			\[
				\E_0(u) \coloneqq 
					\begin{cases}
						\begin{displaystyle}
						\frac{1}{24} \left(\int_{\R} K(z)z^2 dz\right) \int_\R f''(u'(x)) \left|u''(x)\right|^2 dx
						\end{displaystyle} 		&\text{if } u\in X \cap  H^2(\R),  \\
						+\infty 						&\text{otherwise}.
				\end{cases} 
			\]			
	\end{cor}
	
	\begin{proof}
		A change of variables gives
			\[
				\E_h(u) = \int_{\R} K(z) z^2
								\left[	\frac{1}{(hz)^2} \int_{\R}
												f( \left| u'(x)\right| ) -  f \left(\left|\frac{u(x+hz)-u(x)}{hz}\right|\right) dx
								\right] dz.
			\]
		Recalling \eqref{eq:Eh},
		we notice that the quantity between square brackets is equal to $E_{hz}(u')$,
		therefore the conclusion follows by a straightforward adaptation of the proof of Theorem \ref{stm:1D-Gconv}
		(see also the proof of Proposition \ref{stm:intermediate}).
	\end{proof}

Corollary \ref{stm:1D-cor} concludes the analysis when $d=1$,
so we may henceforth assume that $d\geq 2$.
Our aim is proving that
the restrictions to $X$  of the functionals $\E_h$ $\Gamma$-converge
w.r.t. the $H^1(\Rd)$-topology to $\E_0$.
The gist of our proof is a slicing procedure,
which amounts to express the $d$-dimensional  energies $\E_h$ 
as superpositions of the $1$-dimensional energies $E_h$,
regarded as functionals on each line of $\Rd$.

Hereafter we tacitly assume that 
$\Omega$, $K$, and $f$ satisfy the hypotheses made at the beginning of the section.
When $z\in \Rd\setminus\set{0}$, we set
	\[
	\hat{z}^\perp \coloneqq \Set{ \xi \in \Rd : \xi \cdot \hat{z} = 0}.
	\]

\begin{lemma}[Slicing] \label{stm:slicing}
	For $u\in X$, $z\in \Rd\setminus\set{0}$, and $\xi \in \hat{z}^\perp$,
	we define $w_{\hat z,\xi}\colon \R \to \R$ as $w_{\hat z,\xi}(t)\coloneqq u(\xi + t\hat{z})$.
	Then, $w'_{\hat z,\xi}(t) = \nabla u(\xi + t\hat{z}) \cdot \hat{z}$ and
	\begin{equation}\label{eq:slicing}
	\E_h(u) = \int_{\Rd}\int_{z^\perp} K(z) \left| z \right|^2  E_{ h\left| z \right| } (w'_{\hat z,\xi}) d \H^{d-1}(\xi) dz,
	\end{equation}
	where $E_{h \left| z \right|}$ is as in \eqref{eq:Eh}
	(note that the function $f$ in \eqref{eq:Eh} must be replaced here
	by $f(\left| t \right|)$).
\end{lemma}	
\begin{proof}
	Formula \eqref{eq:slicing} is an easy consequence of Fubini's Theorem.
	Indeed, once the direction $\hat{z}\in\mathbb{S}^{d-1}$ is fixed,
	we can write $x\in\Rd$ as $x = \xi + t\hat{z}$
	for some $\xi \in \Rd$ such that $\xi\cdot z = 0$ and $t\in\R$.
	Using this decomposition, we have
		\[\begin{split}
			\F_h(u) = &\int_{\Rd}\int_{\Rd} K(z) f\left(  \frac{ \left| u(x+hz)-u(x) \right| }{ h \left| z \right|}\right) dz dx \\
						= & \int_{\Rd}\int_{z^\perp} \int_{\R}
									K(z) f\left( 
										\frac{ \left| w_{\hat z,\xi}(t + h \left| z \right| ) - w_{\hat z,\xi}(t) \right| }{ h \left| z \right|}
									\right) dt  d\H^{d-1}(\xi) dz,
		\end{split}\]
	whence
	\[
	\E_h(u) = \frac{1}{h^2}\int_{\Rd}\int_{z^\perp}\int_{\R} K(z) \left[
	f\left( \left| w'_{\hat z,\xi}(t)\right| \right)
	- f\left( \frac{ \left| w_{\hat z,\xi}(t+h|z|) - w_{\hat z,\xi}(t) \right| }{h \left| z \right| }\right)
	\right] dt d\H^{d-1}(\xi) dz.	
	\]
	To obtain \eqref{eq:slicing},
	it now suffices to multiply and divide the integrands by $\left| z \right|^2$.
\end{proof}

The connection with the $1$-dimensional case provided by Lemma \ref{stm:slicing}
suggests that the $\Gamma$-convergence of the functionals $E_h$ might be exploited
to prove Theorem \ref{stm:Gconv}.
Though, to be able to apply the results of Section \ref{sec:1D},
we need the functions $w_{\hat z,\xi}$ in \eqref{eq:slicing}
to admit a second order weak derivative for a.e. $z$ and $\xi$.
This poses no real problem for the proof of the upper limit inequality,
because we may reason on regular functions;
as for the lower limit one, we shall tackle the difficulty in the next subsection
by means of a compactness criterion, see Lemma \ref{stm:cpt} below. 
For the moment being,
we are able to  establish the following:
	
	\begin{prop}\label{stm:intermediate}
		Let $u\in X\cap H^2(\R^d)$. Then:
		\begin{enumerate}
			\item For any family $\set{u_h}\subset X$
				that converges to $u$ in $H^1(\Rd)$, there holds
			\begin{equation*}
						\E_0(u)\le \liminf_{h\searrow 0} \E_h(u_h).
					\end{equation*}
			\item If $u \in X\cap C^3(\Rd)$, then
					\begin{equation*}
						 \E_0(u) = \lim_{h\searrow 0} \E_h(u).
					\end{equation*}
		\end{enumerate}		
	\end{prop}
	
	\begin{proof}
		We prove both the assertions by using the slicing formula \eqref{eq:slicing}.
		\begin{enumerate}
			\item For all $h>0$, $z\in\Rd\setminus\set{0}$, and $\xi \in \hat{z}^\perp$,
				we let $w_{h; \hat z,\xi}\colon \R \to \R$ be defined
				as $w_{h; \hat z,\xi}(t)\coloneqq u_h( \xi + t\hat{z} )$.				
				Then,
					\[
		\E_h(u_h) = \int_{\Rd}\int_{z^\perp} K(z) \left| z \right|^2  E_{ h\left| z \right| } (w'_{h; \hat z,\xi}) d\H^{d-1}(\xi)dz,
					\]
				and, by Fatou's Lemma,
					\begin{equation}\label{eq:liminf-sl}
						\liminf_{h\searrow 0} \E_h(u_h) 
								\geq \int_{\Rd}\int_{z^\perp} K(z) \left| z \right|^2 
						\big[ \liminf_{h\searrow 0} E_{ h\left| z \right| } (w'_{h; \hat z,\xi}) \big] d\H^{d-1}(\xi) dz.
					\end{equation}
				
				Let $ w_{\hat z,\xi}$ be as in Lemma \ref{stm:slicing}.				
				Note that for any kernel $\rho\colon \Rd \to [0,+\infty)$ such that
				$\lVert \rho \rVert_{L^1(\Rd)} = 1$ we may write
					\[\begin{split}
						\int_{\Rd} \left| \nabla u_h - \nabla u  \right|^2
							& \geq \int_{\Rd} \rho( z )
							\int_{\hat{z}^\perp} \int_{\R} 
								\left| \big(\nabla u_h(\xi + t\hat{z}) - \nabla u(\xi + t\hat{z}) \big) \cdot \hat{z}\right|^2
								dt d\H^{d-1}(\xi) dz \\
							& = \int_{\Rd} \rho( z )
									\int_{\hat{z}^\perp} \int_{\R}  \left| w'_{h; \hat z,\xi}(t) - w'_{\hat z,\xi}(t) \right|^2
									dt d\H^{d-1}(\xi) dz.
					\end{split}\]
				Since the left-hand side vanishes as $h\searrow 0$,
				it follows that there exists a subsequence of $\set{w'_{h; \hat z,\xi}}$, which we do not relabel,
				that converges in $L^2(\R)$ to $w'_{\hat z,\xi}$
				for $\mathcal{L}^d$-a.e. $z\in\Rd$ and $\H^{d-1}$-a.e. $\xi\in\hat{z}^\perp$.
				In particular, by assumption, $w'_{\hat z,\xi}\in H^1(\R)$ for a.e. $(z,\xi)$
				and it equals $0$ on the complement of some open interval $I_{\hat z,\xi}$.
				
				From the previous considerations, we see that
				Proposition \ref{stm:1D-Gliminf} can be applied
				on the right-hand side of \eqref{eq:liminf-sl}, yielding
					\[
						\liminf_{h\searrow 0} \E_h(u_h) 
							\geq \int_{\Rd}\int_{z^\perp} K(z) \left| z \right|^2 E_0 (w'_{\hat z,\xi}) d\H^{d-1}(\xi) dz = \E_0(u).
					\]
			\item For any fixed $z\in\Rd\setminus \set{0}$ and $\xi\in\hat{z}^\perp$,
				we define the function $w_{\hat z,\xi}\in C^3(\R)$ as above.
				Since $\Omega$ is bounded, there exists $r>0$ such that,
				for any choice of $z$, $w'_{\hat z,\xi}(t) = \nabla u(\xi + t \hat{z}) \cdot \hat{z} = 0$
				whenever $\xi\in z^\perp$ satisfies $\left| \xi \right| \geq r$,
				while $w'_{\hat z,\xi}(t)$ is supported
				in an open interval $I_{\hat z,\xi}$ if $\left| \xi \right| < r$.
				
				By virtue of the slicing formula \eqref{eq:slicing}, we obtain
					\[
					\left| \E_h(u) - \E_0(u) \right|
							\leq \int_{\Rd}\int_{z^\perp} K(z) \left| z \right|^2  
										\left| E_{ h\left| z \right| } (w'_{\hat z,\xi}) - E_0(w'_{\hat z,\xi})\right|
									d \H^{d-1}(\xi) dz 
					\]
				Proposition \ref{stm:1D-pointlim} gives the existence of 
				a constant $c>0$ and of a continuous, bounded, and increasing function
				$m\colon [0,+\infty) \to [0,+\infty)$ such that $m(0)=0$ and
					\[
						\left| \E_h(u) - \E_0(u) \right|
							\leq  c \int_{\Rd}\int_{z^\perp}
								K(z) \left| z \right|^2 m(h \left| z \right|) d \H^{d-1}(\xi) dz.
					\]
				We remark that here $m$ can be chosen depending only on $\nabla u$,
				and not on $\hat z$ and $\xi$.
				
				Recalling \eqref{eq:K}, to achieve the conclusion
				it now suffices to appeal to Lebesgue's Convergence Theorem.			
		\end{enumerate}
	\end{proof}
	

\subsection{Lower bound, compactness, and proof of the main result}
Similarly to the $1$-dimensional case,
we shall prove the compactness of functions with equibounded energy
by establishing at first a lower bound on the functionals $\E_h$.
More precisely, Lemma \ref{stm:lbound} below shows that, when $f$ is strongly convex,
$\E_h(u)$ is greater than a double integral
which takes into account, for each $z\in \Rd\setminus\set{0}$,
the squared projection of the difference quotients of $\nabla u$ in the direction of $z$.
Thanks to the slicing formula, the inequality follows with no effort
by applying Lemma \ref{stm:1D-enbound} on each line of $\Rd$.

We point out that
our approach results in the appearance of an effective kernel $\tilde K$
in front of the difference quotients.
This function stands as a multidimensional counterpart
of the kernel $J$ in Lemma \ref{stm:1D-enbound};
actually, $\tilde K$ depends both on $K$ and on $J$ 
(see \eqref{eq:tildeK} for the precise definition).
In Lemma \ref{stm:tildeK}, we shall collect some properties of the effective kernel
that will be useful in the proof of Lemma \ref{stm:cpt}.


	\begin{lemma}[Lower bound on the energy]\label{stm:lbound}
		Let us set
			\begin{equation}\label{eq:tildeK}
			\tilde K(z) \coloneqq \int_{-1}^{1} J(r) K_{\left| r \right|}(z) dr
			\qquad\text{for a.e. } z\in\Rd,				
			\end{equation}
		with $J$ as in Lemma \ref{stm:1D-enbound}.
		Then, it holds
			\begin{equation}\label{eq:lbound}
				\E_h(u)\geq 
						\frac{\gamma}{4} \int_{\Rd}\int_{\Rd}
						 \tilde K(z) \left[ \frac{\big( \nabla u(x+ h z) - \nabla u(x)\big)\cdot\hat{z}}{h} \right]^2
						dx dz.
			\end{equation}			
	\end{lemma}
	\begin{proof}
		Thanks to Lemma \ref{stm:slicing}, we can reduce to the $1$-dimensional case, 
		and we take advantage of the lower bound provided by Lemma \ref{stm:1D-enbound}.
		Keeping the notation of Lemma \ref{stm:slicing}, we find
			\[\begin{split}
				\E_h(u)\geq &\frac{\gamma}{4} \int_{\Rd}\int_{z^\perp}
										\int_{\R}\int_{-h\left| z \right|}^{h\left| z \right|} J_{h\left| z \right|}(r) K(z) \left| z \right|^2 
										\left(\frac{w'_{\hat z,\xi}(t+r) - w'_{\hat z,\xi}(t)}{h\left| z \right|}\right)^2
										dr dt d\H^{d-1}(\xi) dz \\
							= & \frac{\gamma}{4} \int_{\Rd}\int_{z^\perp}
										\int_{\R}\int_{-h\left| z \right|}^{h\left| z \right|} J_{h\left| z \right|}(r) K(z) 
										\left(\frac{w'_{\hat z,\xi}(t+r) - w'_{\hat z,\xi}(t)}{h}\right)^2
										drdt\H^{d-1}(\xi) dz.
			\end{split}\]
		To cast this bound in the form of \eqref{eq:lbound},
		we change variables and use Fubini's Theorem:
			\begin{multline*}
				I \coloneqq \int_{\Rd}\int_{z^\perp} \int_{\R}
					\int_{-h\left| z \right|}^{h\left| z \right|} J_{h\left| z \right|}(r) K(z) 
							\left(\frac{w'_{\hat z,\xi}(t+r) - w'_{\hat z,\xi}(t)}{h}\right)^2
					dr dt d\H^{d-1}(\xi) dz \\
				=  \int_{\Rd}\int_{z^\perp} \int_{\R} \int_{-1}^{1} J(r) K(z) 
							\left(\frac{w'_{\hat z,\xi}(t + h\left|z\right| r) - w'_{\hat z,\xi}(t)}{h}\right)^2
						dr dt d\H^{d-1}(\xi) dz \\
				=  \int_{-1}^{0} \int_{\Rd}\int_{z^\perp} \int_{\R}
							J(r) K_{-r}( z ) 
								\left(\frac{w'_{\hat z,\xi}(t - h\left|z\right|) - w'_{\hat z,\xi}(t)}{h}\right)^2
					dt d\H^{d-1}(\xi) dz dr \\
				+ 	\int_{0}^{1} \int_{\Rd}\int_{z^\perp} \int_{\R}
							J(r) K_r(z) 
								\left(\frac{w'_{\hat z,\xi}(t + h\left|z\right|) - w'_{\hat z,\xi}(t)}{h}\right)^2
						dt d\H^{d-1}(\xi) dz dr
			\end{multline*}			
		Note that
			\[\begin{split}
				\int_{-1}^{0} \int_{\Rd}\int_{z^\perp} \int_{\R}
							J(r) K_{-r}( z )
								\left(\frac{w'_{\hat z,\xi}\big( t + h\left|z\right| \big) - w'_{\hat z,\xi}( t )}{h}\right)^2
						dt d\H^{d-1}(\xi) dz dr \\
				= \int_{-1}^{0} \int_{\Rd}\int_{z^\perp} \int_{\R}
							J(r) K_{-r}( z ) 
								\left(\frac{w'_{-\hat z,\xi}\big( -(t + h\left|z\right|) \big) - w'_{-\hat z,\xi}( -t )}{h}\right)^2
						dt d\H^{d-1}(\xi) dz dr \\
				= \int_{-1}^{0} \int_{\Rd}\int_{z^\perp} \int_{\R}
							J(r) K_{-r}( z )
								\left(\frac{w'_{\hat z,\xi}\big( t + h\left|z\right| \big) - w'_{\hat z,\xi}( t )}{h}\right)^2
						dt d\H^{d-1}(\xi) dz dr,
			\end{split}\]
				\comment{ }
			because $w'_{-\hat z,\xi}(-s) = - w'_{\hat z,\xi}(s)$ for all $s\in\R$.
			Thus, we conclude that			
			\begin{multline*}
				I = 
				 \int_{\Rd} \int_{z^\perp} \int_{\R}\left(\int_{-1}^{1} J(r) K_{\left| r \right|}(z) dr\right)					
				\left(\frac{w'_{\hat z,\xi}\big(t + h\left|z\right|\big) - w'_{\hat z,\xi}(t)}{h}\right)^2
							dt d\H^{d-1}(\xi) dz
							\\
							= \int_{\Rd}\int_{\Rd}
						 \tilde K(z) \left[ \frac{\big( \nabla u(x+ h z) - \nabla u(x)\big)\cdot\hat{z}}{h} \right]^2
						dx dz,
			\end{multline*}
			which concludes the proof.
	\end{proof}
	
	Let us remind that, by assumption,
	the kernel $K$ is bounded away from $0$ in a suitable annulus.
	The next lemma shows that
	the effective kernel appearing $\tilde K$ in \eqref{eq:lbound} inherits a similar property.
	
	\begin{lemma}\label{stm:tildeK}
		Let $\tilde K\colon \Rd \to [0,+\infty)$ be as in \eqref{eq:tildeK}.
		Then,
			\begin{equation}\label{eq:summ-tildeK}
				\int_{\Rd} \tilde K(z)\left( 1 + \left| z \right|^2 \right) dz < +\infty.
			\end{equation}
		Moreover, 
		if $\sigma_d$ and $r_1$ are the constants in \eqref{eq:sigma} and \eqref{eq:Kpos}, 
		then,
			\begin{equation}\label{eq:tildeKpos}
				\mathrm{ess} \inf \left\{ \tilde K(z) : z \in B(0,\sigma_d r_1) \right\} > 0.
			\end{equation}
		\end{lemma}
	
		
		\begin{proof}
		The convergence of the integral in \eqref{eq:summ-tildeK} follows easily from \eqref{eq:K}.
		Indeed, by the definition of $\tilde K$, we see that
			\[
				\int_{\Rd} \tilde K(z) dz = \int_{-1}^{1} \int_{\Rd} J(r) K_{\left| r \right|} (z) dz dr
														= \int_{\Rd} K(z) dz;
			\]
		analogously, one finds that
			\[
				\int_{\Rd} \tilde K(z) \left| z \right|^2 dz = c \int_{\Rd} K(z) \left| z \right|^2 dz,
			\]
		for some $c>0$.
		
		For what concerns \eqref{eq:tildeKpos},
		let us set $k\coloneqq \mathrm{ess} \inf \set{ K(z) : z \in B(0,r_1)\setminus B(0,r_0) }$.
		In view of \eqref{eq:Kpos}, $k>0$.
		
		We distinguish between the case $ z\in B(0,r_0) $ and the case $ z\in B(0,r_1)\setminus B(0,r_0) $.
		In the first situation, for a.e. $z\in\Rd$,
			\[\begin{split}
				\tilde K(z) \geq & 2 \int_{\frac{ \left| z \right| }{r_1}}^{\frac{ \left| z \right| }{r_0}} J(r) K_r(z) dr
								\geq  2k \int_{\frac{ \left| z \right| }{r_1}}^{\frac{ \left| z \right| }{r_0}} \frac{1}{r^d}J(r) dr \\
								= & \frac{2k}{ \left| z \right|^{d-1} }
												\int_{r_0 }^{r_1} s^{d-2} \left( 1 - \frac{\left| z \right|}{s}  \right) ds.
			\end{split}\]
		When $ z\in B(0,r_1)\setminus B(0,r_0) $, instead, similar computations get
			\[
				\tilde K(z)	\geq 2 \int_{\frac{ \left| z \right| }{r_1}}^{1} J(r) K_r(z) dr
								= \frac{2k}{ \left| z \right|^{d-1} }
										\int_{\left| z \right| }^{r_1} s^{d-2} \left( 1 - \frac{\left| z \right|}{s}  \right) ds
				\qquad \text{for a.e. } z\in\Rd,
			\]
		so that we obtain
			\begin{equation}\label{eq:est-tildeK}
				\tilde K(z) 
					\geq  \frac{2k}{\left| z \right|^{d-1} }
						\int_{\mathrm{max}(r_0,\left| z \right|) }^{r_1} s^{d-2} \left( 1 - \frac{\left| z \right|}{s}  \right) ds
					\qquad \text{for a.e. } z\in\Rd.
			\end{equation}
			
		When $d=2$, the estimate above becomes
			\[
				\tilde K(z) \geq 2k
					 \left[ \frac{r_1 - \mathrm{max}(r_0,\left| z \right|) }{\left| z \right| }
				 				-\log \left( \frac{r_1}{\mathrm{max}(r_0,\left| z \right|)}\right)
				 	\right]
				 \qquad \text{for a.e. } z\in\Rd.
			\]
		Exploiting the concavity of the logarithm, we see that
		the lower bound that we have obtained is strictly positive
		if $ \left| z \right| < r_1 = \sigma_2 r_1$.
			
		On the other hand,
		putting $M\coloneqq \mathrm{max}(r_0,\left| z \right|)$ for shortness,
		if $d\geq 3$, the right-hand side in \eqref{eq:est-tildeK} equals
			\[
				\frac{2k}{(d-1)(d-2) \left| z \right|^{d-1}}
				\left[ 
					(d-2) \left( r_1^{d-1} - M^{d-1} \right)
					- 
					(d-1)\left| z \right|\left(r_1^{d-2} - M^{d-2}\right)
				\right],
			\]
		and therefore
			\begin{equation*} \label{eq:tildeK>0}\begin{split}
				\tilde K(z) \geq & \frac{2k M^{d-2}}{(d-1)(d-2)\left| z \right|^{d-1}} \\
							& \cdot \left\{
								\left( \frac{r_1}{M}\right)^{d-2}
									\left[ (d-2)r_1 - (d - 1)\left| z \right| \right]
									- 								 
									\left[ (d-2)M - (d - 1)\left| z \right| \right]
							\right\} 
				\end{split}\end{equation*}
		for a.e. $z\in\Rd$. When $\left| z \right| < \frac{d-2}{d-1}r_1 = \sigma_d r_1$,
		the quantity between braces is strictly positive if
			\[
				\frac{(M - \left| z \right| ) d - (2M - \left| z \right| )}{(r_1 - \left| z \right| ) d - (2r_1 - \left| z \right| )}
					< \left( \frac{r_1}{M} \right)^{d-2}.
			\]
		Observe that both the left-hand side and the right-hand one are strictly increasing in $d$;
		also, the left-hand side is bounded above by $(M - \left| z \right| ) / (r_1 - \left| z \right| )$,
		so the last inequality holds if
			\[
			 	\frac{M - \left| z \right|}{r_1 - \left| z \right|} < \frac{r_1}{M},
			\]
		which, in turn, is true for all $z\in B(0,r_1)$.
	\end{proof}

We are now in the position to prove that families with equibounded energy are compact in $H^1(\R^d)$,
and that their accumulation points admit second order weak derivatives.

	\begin{lemma}[Compactness] \label{stm:cpt}
		If $\set{u_h}\subset X$ satisfies $\E_h(u_h)\leq M $ for some $M\geq 0$,
		there exist a subsequence $\set{u_{h_\ell}}$ and a function $u\in X \cap H^2(\R^d)$
		such that $u_{h_\ell}\to u$ in $H^1(\Rd)$.
	\end{lemma}	

	\begin{proof}
          \comment {
                      }
		Let $\tilde k \coloneqq \mathrm{ess} \inf \set{ \tilde K(z) : z \in B(0,\sigma_d r_1) }$;
		Lemma \ref{stm:tildeK} ensures that $\tilde k >0$.
		We consider a function $\rho \in C^\infty_c([0,+\infty))$ such that
			\[
				\rho(r) = 0 \qquad \text{if } r \in \left[ \frac{\sigma_d r_1}{\sqrt{2}},+\infty \right),
			\]
		and we further require that
			\[
				0\leq \rho(r) \leq \tilde k
					\quad\text{and}\quad
				\left|\rho'(r)\right| \leq \tilde k.
			\]
		
		For $h>0$ and $y\in \Rd$, we set 
			\[
				\rho_h(y)\coloneqq \frac{1}{ch^d}\rho\left(\frac{\left| y \right|}{h}\right),
					\qquad \text{with } c\coloneqq \int_{\Rd} \rho(\left| y \right|)dy,
			\]
		and we introduce the functions $v_h\coloneqq \rho_{h} \ast u_h$, as before.
		
		Each function $v_h$ is a smooth function and, for all $\tilde h\in(0,1)$,
		its support is contained in
		\[
			\Omega_{\tilde h} \coloneqq \set{ x : \dist(x,\Omega)\leq 2^{-1/2}\tilde h \sigma_d r_1}
		\]
		if $h\in(0,\tilde h)$.
		In particular, we can choose $\tilde h$ so small that $\partial \Omega_{\tilde h}$ is still Lipschitz.
		For such an $\tilde h$, we assert that
		the family $\set{ v_h }_{h\in(0,\tilde h)}$ is relatively compact in $H^1_0(\Omega_{\tilde h})$.
		In order to prove this, we first remark that
			\begin{equation}\label{eq:hess-lap}
				\int_{\Omega_{\tilde h}} \left| \nabla^2 v_h \right|^2
					= \int_{\Omega_{\tilde h}} \left| \Delta v_h \right|^2,
			\end{equation}
		and next we show that the right-hand side is uniformly bounded. 
		
		We observe that $\int_{\Rd} \nabla \rho_{h}(y)dy = 0$ for all $h>0$,
		because $\rho$ is compactly supported. Hence,
			\[
			\begin{split}
				\lVert \Delta v_h \rVert^2_{ L^2( \Omega_{\tilde h} ) }
								= & \int_{\Rd} \left| \Delta v_h \right|^2 \\
								= &	\int_{\Rd} \left| \int_{\Rd}
												\nabla \rho_{h}(y) \cdot \big( \nabla u_h(x+y) -  \nabla u_h(x) \big)
										dy \right|^2 dx \\
								\leq & \int_{\Rd} \left[\frac{1}{c h^{d+1}}\int_{\Rd}
										 \left| \rho' \left( \frac{\left| y \right|}{h} \right) \right|
											\left|\big( \nabla u_h(x+y) -  \nabla u_h(x) \big) \cdot \hat{y} \right| 
									dy \right] ^2 dx.
			\end{split}\]
		By our choice of $\rho$ and \eqref{eq:tildeKpos}, we find
			\[\begin{split}
				\lVert \Delta v_h \rVert^2_{ L^2( \Omega_{\tilde h} ) }
					\leq & \int_{\Rd} \left[ \frac{1}{c h} \int_{\Rd}
										\tilde K_h ( y )
											\left| \big( \nabla u_h( x+y) -  \nabla u_h(x) \big) \cdot \hat{y} \right|
									dy \right] ^2 dx \\
					\leq & \int_{\Rd} \left[ \frac{1}{c h} \int_{\Rd}
										\tilde K( z )
											\left| \big( \nabla u_h( x + hz ) -  \nabla u_h(x) \big) \cdot \hat{z} \right|
									dz \right] ^2 dx		
			\end{split}\]
		Further, since $\tilde K\in L^1(\Rd)$, Jensen's inequality and Fubini's Theorem yield
			\[
				\lVert \Delta v_h \rVert^2_{ L^2( \Omega_{\tilde h} ) }
					\leq \frac{\lVert \tilde K \rVert_{L^1(\Rd)}}{c^2}
									\int_{\Rd} \int_{\Rd} \tilde K(z)
											\left[ \frac{\big( \nabla u_h( x+ hz) -  \nabla u_h(x) \big) \cdot \hat{z}}{h} \right]^2
										dx dz.
			\]				
		The lower bound \eqref{eq:lbound} entails
			\[
				\lVert \Delta v_h \rVert^2_{ L^2( \Omega_{\tilde h} ) }
						\leq \frac{4}{c^2 \gamma} \lVert \tilde K \rVert_{L^1(\Rd)} \E_h(u_h),
			\]
		so that, in view of the assumption $\E_h(u_h)\leq M$ and of \eqref{eq:hess-lap}, we get
			\begin{equation}\label{eq:nabla2}
				\lVert \nabla^2 v_h \rVert^2_{ L^2( \Omega_{\tilde h} ) }
						\leq \frac{4 M}{c^2 \gamma} \lVert \tilde K \rVert_{L^1(\Rd)} .
			\end{equation}
		
		We argue as in the proof of Lemma \ref{stm:1D-cpt}.
		We recall that,  for $h\in(0,\tilde h)$,
		each $v_h$ vanishes on the complement of $\Omega_{\tilde h}$,
		and thus, by Poincar\'e inequality,
		\eqref{eq:nabla2} implies a uniform bound on the norms
		$\lVert v_h \rVert_{H^2_0(\Omega_{\tilde h})}$.
		As a consequence, by Rellich-Kondrachov Theorem,
		the family $\set{\tilde v_h}_{h\in(0,\tilde h)}$
		of the restrictions of the functions $v_h$ to $\Omega_{\tilde h}$
		admits a subsequence $\set{\tilde v_{h_\ell}}$ that converges in $H^1_0(\Omega_{\tilde h})$
		to a function $\tilde u \in H^2_0(\Omega_{\tilde h})$.
		Actually, the support of $\tilde u$ is contained in $\bar\Omega$, and,
		if we put,	
			\[
				u(x)\coloneqq\begin{cases}
							\tilde u(x) 	&\text{if } x\in \bar\Omega, \\
							0					& \text{otherwise},
						\end{cases}
			\]
		we infer that $\set{v_{h_\ell}}$ converges in $H^1(\Rd)$ to $u \in X\cap H^2(\Rd)$.
		
		To accomplish the proof, it suffices to show that
		the $L^2$ distance between $\nabla u_h$ and $\nabla v_h$ vanishes when $h\searrow 0$.
	Since $\rho_h$ has unit $L^1(\Rd)$-norm and is radial, we have
		\[\begin{split}
			\int_{\Rd} \lvert \nabla v_h(x) - & \nabla u_h(x)\rvert^2 dx 
				= \int_{\Rd} \left| \int_{\Rd} 
									\rho_{h}(y) \big( \nabla u_h(x+y) -  \nabla u_h(x) \big)dy
								\right|^2 dx \\
				= & \frac{1}{4}\int_{\Rd} \left| \int_{\Rd}
									\rho_{h}(y)  \big( \nabla u_h(x+y) + \nabla u_h(x-y) -  2\nabla u_h(x) \big)dy
								\right|^2 dx \\
				\leq & \frac{1}{4}\int_{\Rd} \int_{\Rd}
								\rho_{h}(y) \left| \nabla u_h(x+y) + \nabla u_h(x-y) -  2\nabla u_h(x) \right|^2 dy  dx.
		\end{split}\]
	We remark that for any fixed $y\in\Rd\setminus\set{0}$ and for all $p\in\Rd$,
	the identity
	$ \left| p \right|^2 = \left| p\cdot y \right|^2 + \left| (\mathrm{Id} - y\otimes y) p \right|^2 $
	can be reformulated as
		\begin{equation}\label{eq:proj}\begin{split}
			\left| p \right|^2 	= & \left| p\cdot y \right|^2
													+ \int_{\hat{y}^\perp}
															\pi( \left| \eta \right| )\left| p\cdot \eta\right|^2 d\H^{d-1}(\eta) \\
										= &  \left| p\cdot y \right|^2
													+ \frac{1}{h^2} \int_{\hat{y}^\perp}
															\pi_h(\eta)\left| p\cdot \eta\right|^2 	d\H^{d-1}(\eta),
		\end{split}\end{equation}
	where $\pi\colon [0,+\infty) \to [0,+\infty)$ is a continuous function such that
		\[
			\int_{e_d^\perp} \pi( \left| \eta \right| ) \left| \eta \right|^2 d\H^{d-1}(\eta) = 1,
		\]
	and $\pi_h(\eta)\coloneqq h^{-d+1} \pi( \left| \eta \right| / h)$.
	We further prescribe that
		\begin{equation*}
			\pi(r) = 0 \qquad \text{if } r\in \left[ \frac{\sigma_d r_1}{\sqrt{2}}, +\infty \right)
		\end{equation*}
	and that $\lim_{r\searrow 0} \pi(r) / r \in \R$.
	
	We apply the formula \eqref{eq:proj}
	to $p_h(x,y) \coloneqq \nabla u_h(x+y) + \nabla u_h(x-y) -  2\nabla u_h(x)$
	and we find that
		\begin{equation}\label{eq:nablas}
			\int_{\Rd} \left| \nabla v_h(x) - \nabla u_h(x)\right|^2 dx  \leq \frac{1}{4}\left( I_1 + I_2 \right),
		\end{equation}
	where
		\[
		\begin{gathered}
			 I_1 \coloneqq \int_{\Rd} \int_{\Rd} 
			 									\rho_{h}(y) \left| y \right|^2 
			 										\left| p_h(x,y) \cdot \hat{y} \right|^2 dy  dx, \\
			I_2 \coloneqq \frac{1}{h^2}\int_{\Rd} \int_{\Rd} \int_{\hat{y}^\perp} \rho_{h}(y) \pi_h(\eta)
													\left| p_h(x,y)\cdot \eta \right|^2 d\H^{d-1}(\eta) dy  dx .
		\end{gathered}
		\]
	We first consider $I_1$. 
	Keeping in mind that $\rho$ is compactly supported and
	$\rho(\left| y \right|) \leq \tilde k \leq \tilde K(y)$ for a.e. $y\in B(0, 2^{-1/2}\sigma_d r_1)$,
	we get
		\[
		\begin{split}
			I_1 \leq  \frac{(\sigma_d r_1)^2}{c} 
							& \left[
									\int_{\Rd} \int_{\Rd}
										\tilde K_{h}(y) \left| \big( \nabla u_h(x+y) - \nabla u_h(x) \big) \cdot \hat{y} \right|^2 dy dx
							\right. \\
						& \left. + \int_{\Rd} \int_{\Rd}
										\tilde K_{h}(y) \left| \big( \nabla u_h(x-y)  - \nabla u_h(x) \big)\cdot \hat{y} \right|^2 dy  dx
							\right],
		\end{split}
		\]
	and, by \eqref{eq:lbound},
		\begin{equation}\label{eq:I1}
			I_1 \leq \frac{ 8 (\sigma_d r_1)^2 M}{c \gamma} h^2.
		\end{equation}
			
	As for $I_2$, we assert that there exist a constant $L>0$,
	depending on $d$, $\sigma_d$, $r_1$, $\tilde k$, and $c$, such that
		\begin{equation}\label{eq:I2}
		I_2 \leq \frac{ L M}{\gamma} h^2.
		\end{equation}
	To prove the claim, we write the integrand appearing in $I_2$ as follows:
		\[
		\begin{split}
			p_h(x,y)\cdot \eta = & \big( \nabla u_h(x+y) + \nabla u_h(x-y) -  2\nabla u_h(x) \big)\cdot \eta \\
				= & \big( \nabla u_h(x+y) + \nabla u_h(x-y) - 2 \nabla u_h(x-\eta) \big)\cdot \eta \\
					& + 2 \big( \nabla u_h(x-\eta) - \nabla u_h(x) \big)\cdot \eta \\
				= & \big( \nabla u_h(x+y) - \nabla u_h(x-\eta) \big)\cdot ( \eta + y ) \\
					& + \big( \nabla u_h(x-y) - \nabla u_h(x-\eta) \big)\cdot ( \eta - y ) \\
					& - \big( \nabla u_h(x+y) - \nabla u_h(x-y) \big)\cdot y
					+ 2 \big( \nabla u_h(x-\eta) - \nabla u_h(x) \big)\cdot \eta .
		\end{split}
		\]
	
	We plug this identity in the definition of $I_2$ and we find that
		\[
		\begin{split}
			I_2 \leq & \frac{4}{c} \int_{\Rd} \int_{\Rd} \int_{\hat{y}^\perp} \rho(\left| y \right|) \pi( \left| \eta \right| )
												\left| 
													\big(\nabla u_h(x+hy) - \nabla u_h(x-h\eta) \big)\cdot ( \eta + y )
												\right|^2 d\H^{d-1}(\eta) dy  dx  \\
						& + \frac{4}{c}  \int_{\Rd} \int_{\Rd} \int_{\hat{y}^\perp} \rho(\left| y \right|) \pi(\left| \eta \right|)
												\left|
													\big(\nabla u_h(x-hy) - \nabla u_h(x-h\eta) \big)\cdot ( \eta - y )
												\right|^2 d\H^{d-1}(\eta) dy  dx \\
						& + \frac{8}{c}  \lVert \pi \rVert_{L^1(e_d^\perp)}
								\int_{\Rd} \int_{\Rd} \rho(\left| y \right|) \left| y \right|^2
												\left|
													\big(\nabla u_h(x+hy) - \nabla u_h(x) \big)\cdot \hat y
												\right|^2 dy  dx \\
						& + \frac{16}{c}  \int_{\Rd} \int_{\Rd} \int_{\hat{y}^\perp}
												\rho(\left| y \right|) \pi( \left| \eta \right| )
												\left|
													\big(\nabla u_h(x+h\eta) - \nabla u_h(x) \big)\cdot \eta
												\right|^2 d\H^{d-1}(\eta) dy dx.
		\end{split}
		\]
	We estimate separately each of the contributions on the right-hand side.
	
	Let us set
	$\mathbb{S}^{d-1}_+ \coloneqq \set{ e \in \mathbb{S}^{d-1} : e \cdot e_d > 0}$ and
	$\mathbb{S}^{d-1}_- \coloneqq \set{ e \in \mathbb{S}^{d-1} : e \cdot e_d < 0}$.
	Hereafter, we denote by $L$ any strictly positive constant
	depending only on $d$, $\sigma_d$, $r_1$, and on the norms of $\rho$ and $\pi$.
	
	Taking advantage of the Coarea Formula,
	we rewrite the first addendum as follows:
		\begin{multline*}
			\int_{\Rd} \int_{\Rd} \int_{\hat{y}^\perp} \rho(\left| y \right|) \pi( \left| \eta \right|) \left| 
										\big(\nabla u_h(x+hy) - \nabla u_h(x-h\eta) \big)\cdot ( \eta + y )
									\right|^2 d\H^{d-1}(\eta) dy  dx \\
				= \int_{\Rd}\int_{\Rd} \int_{\hat{y}^\perp} \rho( \left| y \right|) \pi( \left| \eta \right|) \left|
										\big( \nabla u_h(x+h(\eta+y)) - \nabla u_h(x) \big)\cdot ( \eta + y )
									\right|^2 d\H^{d-1}(\eta) dx  dy \\
				= \int_{\mathbb{S}^{d-1}_+} \int_{\Rd} \int_{\R} \int_{e^\perp}
													r^{d-1}\rho( r ) \pi( \left| \eta \right| )
														\left|
															\big( \nabla u_h(x+h(\eta + r e)) - \nabla u_h(x) \big)\cdot ( \eta + r e )
														\right|^2 d\H^{d-1}(\eta) dr dx  d\H^{d-1}(e) \\
				= \int_{\mathbb{S}^{d-1}_+} \int_{\Rd} \int_{\Rd}
														\left| y \right|^2 \left| y\cdot e \right|^{d-1}\rho( \left| y\cdot e \right| )
														\pi\big( \left| (\mathrm{Id} - e\otimes e) y \right| \big) 
														\left| \big( \nabla u_h(x+hy) - \nabla u_h(x) \big)\cdot \hat y \right|^2
													dy dx  d\H^{d-1}(e).
		\end{multline*}
	Similarly, we have
		\begin{multline*}
			\int_{\Rd} \int_{\Rd} \int_{\hat{y}^\perp} \rho(\left| y \right|) \pi( \left| \eta \right|)
							\left|
								\big(\nabla u_h(x-hy) - \nabla u_h(x-h\eta) \big)\cdot ( \eta - y )
							\right|^2 d\H^{d-1}(\eta) dy  dx
					 \\ = \int_{\mathbb{S}^{d-1}_-} \int_{\Rd} \int_{\Rd}
							\left| y \right|^2 \left| y\cdot e \right|^{d-1}\rho( \left| y\cdot e \right| )
									\pi\big( \left|(\mathrm{Id} - e\otimes e) y \right| \big)
									\left| \big( \nabla u_h(x+hy) - \nabla u_h(x) \big)\cdot \hat y \right|^2
							dy dx  d\H^{d-1}(e),
		\end{multline*}
	and thus
		\[\begin{split}
			\int_{\Rd} & \int_{\Rd} \int_{\hat{y}^\perp} \rho(\left| y \right|) \pi( \left| \eta \right|)
							\left| 
								\big(\nabla u_h(x+hy) - \nabla u_h(x-h\eta) \big)\cdot ( \eta + y )
							\right|^2 d\H^{d-1}(\eta) dy  dx \\
				& + \int_{\Rd} \int_{\Rd} \int_{ \hat{y}^\perp } \rho(\left| y \right|) \pi( \left| \eta \right|)
							\left|
								\big(\nabla u_h(x-hy) - \nabla u_h(x-h\eta) \big)\cdot ( \eta - y )
							\right|^2 d\H^{d-1}(\eta) dy  dx \\
				= & \int_{ \mathbb{S}^{d-1}} \int_{\Rd} \int_{\Rd}
								\left| y\cdot e \right|^{d-1} \left| y \right|^2 \rho( \left| y\cdot e \right| )
									\pi\big( \left|(\mathrm{Id} - e\otimes e) y \right|\big)
									 \left| \big( \nabla u_h(x+hy) - \nabla u_h(x) \big)\cdot \hat y \right|^2
							dy dx d\H^{d-1}(e).
				\end{split}\]
			Let us recall that $\rho(r) = \eta (r) = 0$ if $r \notin [0,2^{-1/2}\sigma_d r_1)$,
			whence, for any $e\in \mathbb{S}^{d-1}$,
			the product $\rho( \left| y\cdot e \right| )\pi\big( \left|(\mathrm{Id} - e\otimes e) y \right|\big)$
			vanishes outside the cylinder
				\[
					C_e \coloneqq \set{ y\in\Rd : 
										\left| y\cdot e \right|, \left|(\mathrm{Id} - e\otimes e) y \right| 
											\in [0,2^{-1/2}\sigma_d r_1)
										} \subset B(0,\sigma_d r_1).
				\]
			We therefore see that the last multiple integral equals
				\begin{multline*}
					\int_{ \mathbb{S}^{d-1}} \int_{\Rd} \int_{C_e}
								\left| y\cdot e \right|^{d-1} \left| y \right|^2 \rho( \left| y\cdot e \right| )
								\pi\big( \left|(\mathrm{Id} - e\otimes e) y \right|\big)
								\left| \big( \nabla u_h(x+hy) - \nabla u_h(x) \big)\cdot \hat y \right|^2
							dy dx d\H^{d-1}(e) 
					\\ \leq L \int_{ \mathbb{S}^{d-1}} \int_{\Rd} \int_{C_e}
								\tilde K(y) \left| \big( \nabla u_h(x+hy) - \nabla u_h(x) \big)\cdot \hat y \right|^2
							dy dx d\H^{d-1}(e)
					\\ \leq \frac{L M}{\gamma} h^2 .
				\end{multline*}
		We then obtain
		\begin{multline}\label{eq:1}
			\frac{4}{c} \int_{\Rd} \int_{\Rd} \int_{\hat{y}^\perp} \rho(\left| y \right|) \pi( \left| \eta \right| )
						\left| 
							\big(\nabla u_h(x+hy) - \nabla u_h(x-h\eta) \big)\cdot ( \eta + y )
						\right|^2 d\H^{d-1}(\eta) dy  dx  \\
			+ \frac{4}{c}  \int_{\Rd} \int_{\Rd} \int_{\hat{y}^\perp} \rho(\left| y \right|) \pi(\left| \eta \right|)
						\left|
							\big(\nabla u_h(x-hy) - \nabla u_h(x-h\eta) \big)\cdot ( \eta - y )
						\right|^2 d\H^{d-1}(\eta) dy  dx \\
			\leq  \frac{ LM }{ \gamma} h^2.
		\end{multline}
		Next, we have
			\begin{align}
				 \frac{8}{c}  \lVert \pi \rVert_{L^1(e_d^\perp)}
						\int_{\Rd} \int_{\Rd} \rho(\left| y \right|) \left| y \right|^2
								\left|	\big(\nabla u_h(x+hy) - \nabla u_h(x) \big)\cdot \hat y \right|^2
							dy  dx 
					& \leq \frac{L M}{\gamma} h^2, \label{eq:2} \\
				\frac{16}{c}  \int_{\Rd} \int_{\Rd} \int_{\hat{y}^\perp}
							\rho(\left| y \right|) \pi( \left| \eta \right| )
								\left| \big(\nabla u_h(x+h\eta) - \nabla u_h(x) \big)\cdot \eta \right|^2
						d\H^{d-1}(\eta) dy dx
					& \leq \frac{L M}{\gamma} h^2. \label{eq:3}			
			\end{align}
		The bound in \eqref{eq:2} may be deduced as the one in \eqref{eq:I1},
		so, to establish \eqref{eq:I2}, we are only left to prove \eqref{eq:3}.
		To this aim, let $\psi\in C^\infty_c(\Rd\times\Rd)$ be a test function.
		By a standard argument and Fubini's Theorem we have that
			\[\begin{split}
				\int_{\Rd} \int_{\hat{y}^\perp}
					& \rho(\left| y \right|) \pi( \left| \eta \right| ) \psi(y,\eta)  d\H^{d-1}(\eta) dy
				\\ = & \lim_{\epsilon\searrow 0} 
					\int_{\Rd}\int_{\Rd} \frac{\left| y \right|}{2\eps} \chi_{\set{ t < \eps}}(\left| \eta \cdot y\right|)
						\rho(\left| y \right|) \pi( \left| \eta \right| ) \psi(y,\eta)  d\eta dy
				\\ = & \lim_{\epsilon\searrow 0}
					\int_{\Rd} \frac{\pi( \left| \eta \right| )}{\left| \eta \right| }
					\left(\int_{\Rd} \frac{\left| \eta \right| }{2\eps} \chi_{\set{ t < \eps}}(\left| \eta \cdot y\right|)
						\rho(\left| y \right|) \left| y \right|  \psi(y,\eta)  dy
					\right) d\eta
				\\ = &  \int_{\Rd} \int_{\hat{\eta}^\perp}
								\frac{\pi( \left| \eta \right| )}{\left| \eta \right| }
								\rho(\left| y \right|) \left| y \right| \psi(y,\eta) 
							d\H^{d-1}(y) d\eta
			\end{split}\]
		(recall that we assume $\lim_{r\searrow 0} \pi(r) / r$ to be finite).
		It follows that
			\[\begin{split}
				\int_{\Rd} \int_{\Rd} & \int_{\hat{y}^\perp}
						\rho(\left| y \right|) \pi( \left| \eta \right| )
						\left| \big(\nabla u_h(x+h\eta) - \nabla u_h(x) \big)\cdot \eta \right|^2
					d\H^{d-1}(\eta) dy dx
				\\ = & \int_{\Rd} \int_{\Rd} \int_{\hat{\eta}^\perp}
						\frac{\pi( \left| \eta \right| )}{\left| \eta \right| }
						\rho(\left| y \right|) \left| y \right| 
						\left| \big(\nabla u_h(x+h\eta) - \nabla u_h(x) \big)\cdot \eta \right|^2
					d\H^{d-1}(y) d\eta dx
				\\ \leq & L \int_{\Rd} \int_{\Rd}
									\tilde K(\eta)\left| \big(\nabla u_h(x+h\eta) - \nabla u_h(x) \big)\cdot \eta \right|^2
								d\eta dx.
			\end{split}\] 
		In view of the bound on the energy, we retrieve \eqref{eq:3}.
		
		The proof is now concluded,
		because from \eqref{eq:nablas}, \eqref{eq:I1}, and \eqref{eq:I2}
		we obtain
			\[
				\int_{\Rd} \left| \nabla v_h(x) - \nabla u_h(x)\right|^2 dx  \leq \frac{L M}{\gamma} h^2,
			\]
		as desired.
	\end{proof}

\begin{rmk}\label{remcrit}
	The choice $u_h=u$ in Lemma \ref{stm:cpt} provides a criterion for a function in $H^1(\Rd)$
	to belong to $H^2(\R^d)$. Namely, when $\Omega$, $K$, and $f$ fulfil the assumptions of the current section
	and $f''$ is bounded,
	a function $u\in X$ is in $H^2(\R^d)$
	if and only if $\E_h(u)\le M$ for some $M>0$ and for all $h$'s small enough.
	One implication is a byproduct of Lemma \ref{stm:cpt},
	while the other follows
	by exploiting the slicing formula and Remark \ref{rmk:f''bounded}:
	indeed, if $f''\leq c$ one finds
		\[
			\E_h(u) \leq
				\frac{c}{2} \left(\int_{\Rd} K(z) \left| z \right|^2 dz\right)
					\int_{\Rd} \left| \nabla^2 u(x)\right|^2 dx.
		\]
\end{rmk}

We can now accomplish the proof of Theorem \ref{stm:Gconv}.

\begin{proof}[Proof of Theorem \ref{stm:Gconv}]
Lemma \ref{stm:cpt} provides the compactness result
of statement \eqref{stm:cptgen} in Theorem \ref{stm:Gconv}.

Turning to the lower limit inequality, 
for any $u\in X$ and for any family $\set{u_h} \subset X$ that converges to $u$ in $H^1(\Rd)$,
we may focus on the situation
when there exists $M\geq 0$ such that $\E_h(u_h) \leq M$ for all $h>0$.
In view of Lemma \ref{stm:cpt},
we have that $u\in H^2(\R^d)$,
thus statement \eqref{stm:Gliminf} follows by
Proposition \ref{stm:intermediate}.

For what concerns the upper limit inequality,
we reason as in the $1$-dimensional case
(see the proof of Proposition \ref{stm:1D-pointlim}).
In order to adapt the argument,
we observe that,
if $u\in X \cap H^2(\Rd)$, by mollification,
we can construct a sequence $\set{u_\ell} \subset X$ of smooth functions
that tend to $u$ in $H^2(\Rd)$
and satisfy $\lim_{\ell \nearrow +\infty} \E_0(u_\ell) = \E_0(u)$,
provided that $f''$ is bounded
or $u\in X \cap H^2(\Rd) \cap W^{1,\infty}(\Rd)$.
Indeed, when one of these assumptions holds,
there exists $c>0$ such that
$f''(\left|\nabla u_\ell(x) \cdot \hat{z}\right|)\leq c$ for a.e. $x$ and all $z$,
and Lebesgue's Theorem applies.
Then, we can establish the upper limit inequality
by combining the approximation by smooth functions
and Proposition \ref{stm:intermediate}.
\end{proof}

We conclude with a couple of remarks.

\begin{rmk}\label{remscala2}
	As in Remark \ref{remscala1},
	we see that	the $\Gamma$-limit of 
	\[
	h \E_h(u) =  \frac{\F_0(u)-\F_h(u)}{h}
	\]
	in $H^1(\Rd)$ is $0$.
	The same $\Gamma$-limit is found
	if one considers the $L^2(\R^d)$-topology on $X$,
	because $h\E_h(u)\geq 0$ for all $u\in X$ and Proposition \ref{stm:intermediate} provides a constant recovery sequence for smooth functions.
\end{rmk}	

\begin{rmk}
Statements \eqref{stm:Gliminf}, \ref{stm:Glimsup-a}, and \ref{stm:Glimsup-b} in Theorem \ref{stm:Gconv},
that is, the $\Gamma$-convergence result, 
are not affected if we replace $X$ with $H^1(\R^d)$; the proof remains essentially the same.
On the other hand,
if we substitute $\Omega$ with $\R^d$, 
the compactness provided by statement \eqref{stm:cptgen} of Theorem \ref{stm:Gconv} may fail.
\end{rmk}


\end{document}